\documentclass[11pt]{article}
\usepackage{amsfonts,amsmath,amssymb,amsthm}
\usepackage{color}
\usepackage{epsfig}
\usepackage{eso-pic}
\usepackage{graphicx}
\usepackage{type1cm}
\usepackage[utf8x]{inputenc} 

\usepackage{dsfont}
\usepackage{mathrsfs}
\usepackage{yfonts}

\usepackage[bottom]{footmisc}
\usepackage{cite}
\usepackage{longtable}

\usepackage{multirow}
\usepackage{subfigure}
\usepackage{graphicx}
\usepackage{graphicx}
\usepackage{graphicx,subfigure,float}

\newcommand*{\Zplus}{\mathbb{Z}_{+}}

\newcommand*{\Pol}{\mathbb{P}}

\newcommand*{\charfunc}{\mathds{1}}

\newtheorem{theorem}{Theorem}

\newtheorem{corollary}[theorem]{Corollary}
\newtheorem{definition}[theorem]{Definition}
\newtheorem{lemma}[theorem]{Lemma}

\newlength{\totalhmargin}
\newlength{\totalvmargin}
\DeclareMathOperator{\sign}{sgn}

\begin{document}

\title{Spectral Galerkin boundary element methods for high-frequency sound-hard scattering problems}

\author{
Akash Anand\thanks{Department of Mathematics and Statistics, Indian Institute of Technology Kanpur, 327 Faculty Building, Kanpur, UP 208016, India. akasha@iitk.ac.in},
Yassine Boubendir\thanks{New Jersey Institute of Technology, Department of Mathematical Sciences, University Heights, Newark NJ 07102, USA. boubendi@njit.edu}, 
Fatih Ecevit\thanks{Bo\u{g}azi\c{c}i University, Department of Mathematics, Bebek TR 34342, Istanbul, Turkey. fatih.ecevit@boun.edu.tr} \thanks{Corresponding author},
and
Souaad Lazergui\thanks{New Jersey Institute of Technology, Department of Mathematical Sciences, University Heights, Newark NJ 07102, USA. sl584@njit.edu}
}

\date{}
\maketitle

\begin{abstract}
This paper is concerned with the design of two different classes of Galerkin boundary element methods for the solution of high-frequency
sound-hard scattering problems in the exterior of two-dimensional smooth convex scatterers. Both methods require a small increase in the
order of $k^{\epsilon}$ (for any $\epsilon >0$) in the number of degrees of freedom to guarantee frequency independent precisions with
increasing wavenumber $k$. In addition, the accuracy of the numerical solutions are independent of frequency provided sufficiently many 
terms in the asymptotic expansion are incorporated into the integral equation formulation. Numerical results validate our theoretical findings.

\end{abstract}

\section{Introduction} \label{sec:1}

Wave propagation simulations at high frequencies require appropriate design of numerical methods. Generally speaking,
classical approaches based on finite elements~\cite{HesthavenWarburton04,DaviesEtAl09,Boffi10},
integral equations~\cite{BrunoKunyansky01,AminiProfit03,BanjaiHackbusch05,TongChew10,BrunoEtAl13},
and finite differences~\cite{NamburuEtAl04,ShererVisbal05,TafloveHagness05} demand discretization in the order of the
wavelength which produce very large linear systems for high wavenumbers. Therefore they are not suitable in the high
frequency regime because of the required computational cost. This is the main reason why several research projects geared
towards the design and analysis of high-frequency simulation strategies were initiated. For instance, in the case of sound-soft scattering
problems, several methods were introduced in the context of single and multiple scattering configurations
\cite{BrunoEtAl04,BrunoEtAl05,BrunoGeuzaine07,HuybrechsVandewalle07,GiladiKeller04,DominguezEtAl07,Ecevit18,EcevitOzen17,EcevitEruslu19}.
Most of these methods were materialized
thanks to the high-frequency asymptotic expansion (ansatz) of the normal derivative of the total field, derived by Melrose
and Taylor in the well known paper \cite{MelroseTaylor85}, for the Dirichlet boundary value problem. Using the asymptotic
expansion of the total field corresponding to the Neumann problem, given in the same paper \cite{MelroseTaylor85}, here
we propose new high-frequency Galerkin boundary methods for sound-hard scattering problems for smooth convex obstacles.

For the Dirichlet boundary value problem, most of the aforementioned techniques use the Melrose-Taylor ansatz. These include the
localized integration based Nystr\"{o}m scheme proposed for single \cite{BrunoEtAl04,BrunoGeuzaine07} and multiple scattering
problems \cite{BrunoEtAl05} (for the derivation of multiple scattering ansatz see \cite{EcevitReitich09,AnandEtAl10}), collocation
technique depending on the numerical steepest descent method \cite{HuybrechsVandewalle07}, and the Galerkin boundary element
methods \cite{DominguezEtAl07,EcevitOzen17,EcevitEruslu19}. The algorithms developed in
\cite{BrunoEtAl04,BrunoGeuzaine07,BrunoEtAl05,HuybrechsVandewalle07} are not supported with convergence analyses, and those
in \cite{BrunoEtAl04,BrunoGeuzaine07,BrunoEtAl05,HuybrechsVandewalle07,DominguezEtAl07} approximate the solution by zero in
the deep shadow region which, as is well known, is true only in the high-frequency limit. In the case of the Dirichlet problem, this
approximation does not effectively impair the accuracy of the numerical solution since the solution in that region rapidly decays
with increasing wavenumber. However, the solution related to the Neumann case decays comparatively slower than the one for the
Dirichlet problem (see Figures~\ref{fig:Dirichlet-NeumannN} and \ref{fig:Dirichlet-NeumannD}), and therefore they may also loose
accuracy for moderate frequencies. The Galerkin boundary element methods proposed in \cite{EcevitOzen17,EcevitEruslu19}
address this problem. For both algorithms, an increase of $\mathcal{O}(k^{\epsilon})$ (for any $\epsilon>0$) in the number of
degrees of freedom is sufficient to fix the approximation error with increasing wavenumber $k$.

In the last decades substantial interest has grown towards high-frequency problems in the fields of mathematical and numerical analysis
\cite{HuybrechsVandewalle07,DominguezEtAl07,EcevitOzen17,EcevitEruslu19,Chandler-WildeGraham09,Chandler-WildeEtAl12,Chandler-WildeEtAl15,ChandlerWildeLangdon06,Chandler-WildeEtAl07,GibbsEtAl19,GrothEtAl18,GrothEtAl13,Hewett15,HewettEtAl15,HewettEtAl13,LangdonEtAl10}.
Indeed, as emphasized above, the design of numerical methods for these problems is based on the use of the ansatz
directly in the numerical scheme. This ansatz was derived analytically using mathematical tools such as pseudo-differential operators,
asymptotic analysis and many others. This gives rise to challenging difficulties mainly related to the development of stable and
convergent numerical algorithms. The aim of this paper is the design and analysis of new Galerkin boundary element methods
for the solution of high-frequency sound-hard scattering problems. From an analytical point of view, this requires a careful analysis
of the Melrose-Taylor ansatz for the Neumann boundary value problem which is more complicated than its Dirichlet counterpart.
The derivation of the ansatz related to the Dirichlet case is provided in complete detail in \cite{MelroseTaylor85} which is not the
case for the Neumann problem. Therefore, for the sake of our analysis, here in this paper we complete the missing parts in this
derivation, and also explain how the ansatz extends over the entire boundary of the scatterers.

The design and numerical analysis of the Galerkin boundary element methods proposed in this paper are based on a careful analysis
of the asymptotic properties of the ansatz
\begin{equation} \label{eq:ansatzintro}
	\eta^{\rm slow} \sim \sum_{p,q,r,\ell} a_{p,q,r,\ell}
\end{equation}
where $\eta^{\rm slow}$ is the unknown of the integral equation formulations as will be explained in \S\ref{sec:2}. We first
determine the H\"{o}rmander classes and obtain wavenumber explicit estimates on the derivatives of the terms $a_{p,q,r,\ell}$. Then, with the
aid of this analysis, we derive sharp wavenumber explicit estimates on the derivatives of the envelope $\eta^{\rm slow}$. Finally, we
use these estimates in the optimal design and numerical analysis of two different classes Galerkin boundary element methods. As we will show,
these methods are capable of delivering prescribed accuracies with the utilization of numbers of degrees of freedom that need to increase
in the order of $k^\epsilon$ (for any $\epsilon > 0$) with increasing wavenumber, and are therefore \emph{almost} frequency independent.

Both numerical methods developed here are frequency independent when an adequate number of terms $a_{p,q,r,\ell}$ are used in the integral formulations.
In this connection, let us mention that the error estimates for the Galerkin boundary element methods developed for the Dirichlet
\cite{DominguezEtAl07,EcevitOzen17,EcevitEruslu19} and the Neumann problem in this paper use C\'{e}a's lemma \cite{Cea64} which leads
to an important factor expressed as a function of the wavenumber. 
In the Dirichlet case, this factor was shown to be $\mathcal{O}(k^{\frac{1}{3}})$ as $k \to \infty$ for the combined field
and star combined integral equations \cite{Galkowski19,GalkowskiEtAl19} (see also \cite{Ecevit18,EcevitEruslu19}), and therefore the Galerkin
approximations are bound to degrade with increasing frequency; this problem was addressed in \cite{Ecevit18} where, in the context of the
combined field and star combined integral equations, it is shown that incorporation of the leading order term in the Dirichlet ansatz
is sufficient to render the methods \cite{EcevitOzen17,EcevitEruslu19} frequency independent. As we will see, the same approach applied to the Neumann
problem will lead to similar observations and results.

To the best of our knowledge, there is no known coercive direct formulation for the Neumann boundary value problem.
However the recent development of a coercive formulation in its indirect form \cite{BoubendirTurc13} suggests that this kind of formulations can also be derived.  
Assuming the existence of such formulations, we rigorously determine the minimal number of terms in the Neumann ansatz \eqref{eq:ansatzintro} that must be incorporated
into the integral equation in order to cancel the effect of the aforementioned factor when C\'{e}a's lemma is used. In this connection we show that,
for any given $\beta \ge 0$, there exists and optimal
collection $\mathcal{F}_{\beta}$ of $(p,q,r,\ell)$ so that the derivatives of the difference
\begin{equation} \label{eq:slowintro}
	\rho^{\rm slow}_{\beta} = \eta^{\rm slow} - \sigma^{\rm slow}_{\beta} = \eta^{\rm slow} - \sum_{(p,q,r,\ell) \in \mathcal{F}_\beta} a_{p,q,r,\ell}
\end{equation}
grow the slowest with increasing $k$. Following this, we use these estimates to show that the incorporation of $\sigma^{\rm slow}_{\beta}$
into any given appropriate integral equation formulation leads to the factor $k^{-\beta/3}$ in the error estimates.

Although the numerical methods we develop in this paper, namely the \emph{frequency-adapted $\beta$-asymptotic Galerkin
boundary element method} and the \emph{$\beta$-asymptotic Galerkin boundary element method based on frequency dependent changes
of variables}, use the same constructions as in their sound-soft versions proposed in \cite{EcevitOzen17} and \cite{EcevitEruslu19} respectively
(see also \cite{Ecevit18}), the analyses are significantly different. The former method resolves the boundary layers
around the shadow boundaries by adequate utilization of subregions in these regions with respect to the frequency. On each subregion, the method uses algebraic polynomials weighted by the oscillations in the incident field of radiation
as in the Dirichlet case \cite{EcevitOzen17}. Similarly, as in \cite{EcevitEruslu19}, the latter method utilizes frequency dependent
changes of variables to resolve the boundary layers around the shadow boundaries. In addition, we show that, for $\beta = 0$
(which means no term $a_{p,q,r,\ell}$ in the ansatz is incorporated into the integral equation formulation) both methods
require only a small increase of size $\mathcal{O}(k^{\epsilon})$ (for any $\epsilon > 0$) in the number of degrees of freedom
to maintain accuracy with increasing $k$. Moreover, as mentioned above, we demonstrate that the methods are frequency independent when sufficiently many terms
in the ansatz are appropriately used in the integral equation formulations.

The paper is organized as follows. In \S\ref{sec:2}, we introduce the sound-hard scattering problem, and discuss the similarities and
differences between the Neumann and Dirichlet high-frequency solutions. In \S\ref{sec:3}, we determine the H\"{o}rmander classes of the terms $a_{p,q,r,\ell}$
in the ansatz \eqref{eq:ansatzintro} and the envelopes $\rho^{\rm slow}_{\beta}$ in \eqref{eq:slowintro}, and derive sharp wavenumber dependent
estimates on their derivatives. We use these estimates in the construction and numerical analysis of the
\emph{frequency-adapted $\beta$-asymptotic Galerkin boundary element method} and the
\emph{$\beta$-asymptotic Galerkin boundary element method based on frequency dependent changes of variables} in \S\ref{sec:4}.
We present numerical results confirming our theoretical findings in \S\ref{sec:5}.
Finally, the derivation of the ansatz is presented in Appendix~\ref{sec:asymptotics} where we also show how it extends to the entire boundary of the scatterers.

\section{Problem statement} \label{sec:2}

We consider the sound-hard scattering problem in the exterior of a smooth, compact and strictly convex obstacle $K \subset \mathbb{R}^2$ illuminated
by a plane wave incidence $u^{\rm inc}(x)=e^{ik\alpha \cdot x}$ with direction $\alpha$, $|\alpha| =1$ and $k >0$. The unknown
scattered field $u$ satisfies \cite{ColtonKress92,Chandler-WildeEtAl12}
\begin{equation}
\label{eq:Helmholtz}
	\left\{
		\begin{array}{l}
			(\Delta +k^2) u = 0
			\quad \text{in } \mathbb{R}^2 \backslash K,
			\\
			\partial_{\nu} u = -\partial_{\nu} u^{\rm inc}
			\quad \text{on } \partial K,
			\\
			\lim_{r \to \infty} \sqrt{r} \big( \frac{\partial u}{\partial r} -iku \big) =0,
			\quad r=|x|,
		\end{array}
	\right.
\end{equation}
where $\nu$ is the exterior unit normal to $\partial K$. The direct approach in high-frequency integral equation
formulations transforms the scattering problem \eqref{eq:Helmholtz} into the computation of the (unknown)
total field, denoted here by $\eta$, on $\partial K$.
In this case, the scattered field can be expressed as the double layer potential
\cite{ColtonKress92,Chandler-WildeEtAl12}
\begin{equation} \label{eq:Dir-trace}
	u(\alpha,x,k) = \int_{\partial K} \dfrac{\partial G_k(x,y)}{\partial \nu(y)} \, \eta(\alpha,y,k) \, ds(y)
\end{equation}
where
\[
	G_k(x,y) = \dfrac{i}{4} \, H^{(1)}_0(k|x-y|)
\]
is the outgoing Green's function for the Helmholtz equation and $H^{(1)}_0$ is the Hankel function of type
one and order zero. One of the primary motivations for using direct formulations is the observation
that $\eta$ is amenable to phase extraction 
\begin{equation} \label{eq:phaseextract}
	\eta(\alpha,x,k) = e^{ik \alpha \cdot x} \, \eta^{\rm slow}(\alpha,x,k)
\end{equation} 
as it is the case for the Dirichlet boundary value problem where the unknown represents the normal derivative of the total field
(see Figures~\ref{fig:Dirichlet-NeumannN} and \ref{fig:Dirichlet-NeumannD}). In this paper, we develop efficient Galerkin boundary element methods
by using the asymptotic behavior (as $k \to \infty$) of the envelope $\eta^{\rm slow}$ in the construction of Galerkin approximation spaces.
This approach is similar to the one used for the sound-soft scattering problem. In Figures~\ref{fig:Dirichlet-NeumannN} and \ref{fig:Dirichlet-NeumannD},
we display the total field and the normal derivative of the total field respectively for the Neumann and Dirichlet boundary value problems.
As we can see, these densities have similar asymptotic characteristics. Specifically, they both posses boundary layers around the shadow boundaries,
and decay rapidly in the deep shadow region with increasing wavenumber. From a numerical perspective, however, approximating the density related to
the Neumann problem is more challenging since its slow part ($\eta^{\rm slow}$) oscillates more strongly around the shadow boundaries
and decays indubitably slower in the shadow region. Consequently, obtaining highly accurate numerical approximations to the Neumann problem
is significantly more challenging when compared to the Dirichlet case.

\begin{figure}[H]
	\centering
	\subfigure{
	\includegraphics*[width=5.8in,viewport=85 6 820 415,clip]{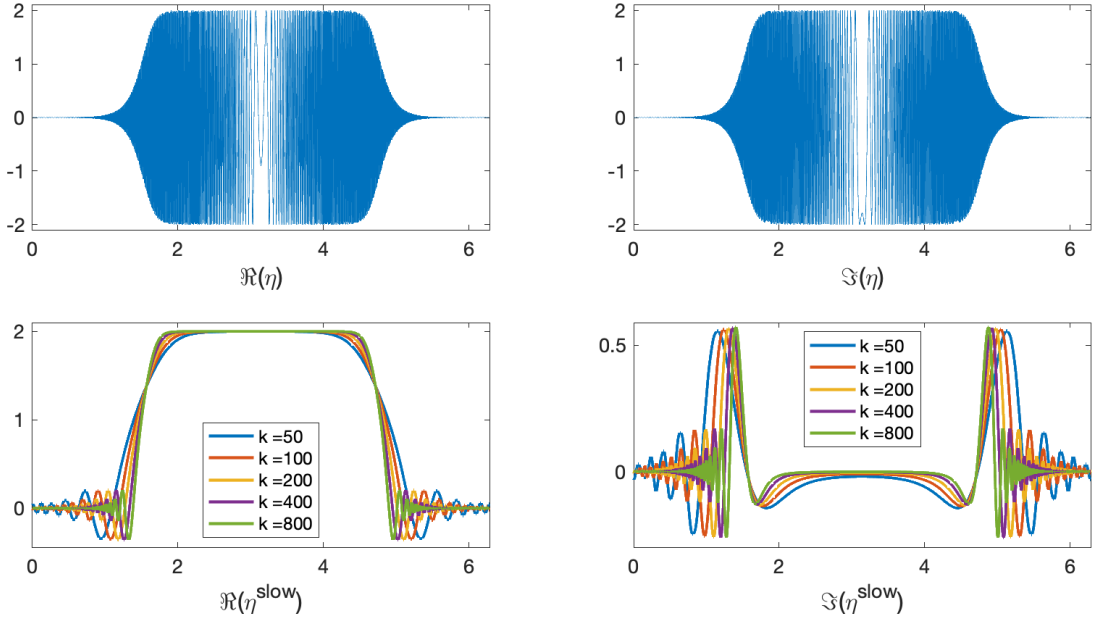}}		
	\caption{Plots of the real and imaginary parts of the total field (top row) and the slow envelope (bottom row)
	for the sound-hard scattering problem in the case of a plane wave incidence with direction $\alpha = (1,0)$ impinging on the unit circle
	$(\cos t, \sin t)$ for $k=50,100,200,400,800$.}
	\label{fig:Dirichlet-NeumannN}
\end{figure}

\begin{figure}[H]
	\centering
	\subfigure{
	\includegraphics*[width=5.8in,viewport=85 6 820 415,clip]{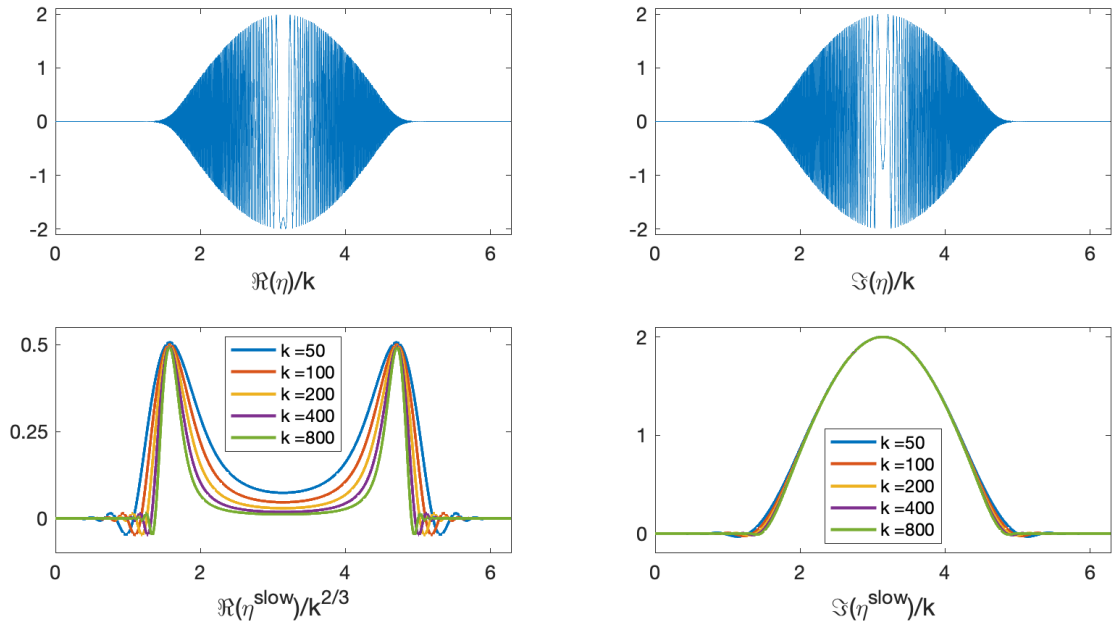}}
	\caption{Plots of the real and imaginary parts of the normal derivative of the total field modulated by $k$ (top row)
	and the slow envelope modulated by $k^{2/3}$ and $k$ (bottom row) for the sound-soft
	scattering problem in the case of a plane wave incidence with direction $\alpha = (1,0) $ impinging on the unit circle
	$(\cos t, \sin t)$ for $k=50,100,200,400,800$.}
	\label{fig:Dirichlet-NeumannD}
\end{figure}

\section{H\"{o}rmander classes and wavenumber explicit derivative estimates} \label{sec:3}

This section is dedicated to the study of the asymptotic expansion of the envelope $\eta^{\rm slow}$ defined in \eqref{eq:phaseextract}.
We first observe that the incident plane wave $u^{\rm inc}(x) = e^{ik \alpha \cdot x}$ determines the illuminated and shadow regions,
and the shadow boundaries on $\partial K$ as
\begin{align*}
	\partial K^{IL} & = \{ x \in \partial K: \alpha \cdot \nu(x) < 0 \} \\
	\partial K^{SR} & = \{ x \in \partial K: \alpha \cdot \nu(x) > 0 \} \\
	\partial K^{SB} & = \{ x \in \partial K: \alpha \cdot \nu(x) = 0 \}.
\end{align*}
We let $2P = |\partial K|$, and we choose $\gamma$ as the $2P$-periodic arc length parameterization of the boundary
$\partial K$ in the counterclockwise direction such that the shadow boundaries $\partial K^{SB} = \gamma \left( \left\{ t_1,t_2 \right\} \right)$
are determined by the parameters $0 < t_{1} < t_{2} < 2P$ satisfying $t_1+t_2 = 2P$, and the illuminated and shadow regions are given by 
$ \partial K^{IL} = \gamma \left( \left( t_1,t_2 \right) \right)$ and $ = \partial K^{SR} = \gamma \left( \left( 0,t_1 \right)
\cup \left( t_2,2P \right) \right)$. In what follows we shall write $\partial K^{SB}$ for $\left\{ t_1,t_2 \right\}$,
$\partial K^{IL}$ for $(t_1,t_2)$, and $\partial K^{SR}$ for $( 0,t_1) \cup (t_2,2P)$.
For convenience, we shall also write $\eta(s,k)$, $\eta^{\rm slow}(s,k)$, $\nu(s)$, etc. rather than $\eta(\alpha,\gamma(s),k)$, $\eta^{\rm slow}(\alpha,\gamma(s),k)$,
$\nu(\gamma(s))$, etc. where $\alpha$ is eliminated and $\gamma(s)$ is replaced by $s$. In the next theorem, we present the asymptotic behavior of 
$\eta^{\rm slow}$ in a two-dimensional setting while a general version is given in Theorem~\ref{thm:asymp-exp-global} of
Appendix~\ref{sec:asymptotics}.

\begin{theorem} \label{thm:asymp-exp}
The envelope $\eta^{\rm slow}$ belongs to the H\"{o}rmander class $S^{0}_{\frac{2}{3},\frac{1}{3}} ([0,2P]\times(0,\infty))$ and admits the
asymptotic expansion
\begin{equation}
\label{eq:MT85Neumann}
	\eta^{\rm slow}(s,k)
	\sim
	\sum_{p,q,r \in \mathbb{Z}_+ \atop \ell \in -\mathbb{N}}
	a_{p,q,r,\ell} (s,k)
\end{equation}
with
\[
	a_{p,q,r,\ell} (s,k)
	= k^{-\frac{1+2p+3q+r+\ell}{3}+(\ell+1)_{-}} \,
	b_{p,q,r,\ell} (s) \,
	(\Psi^{r,\ell})^{(p)}(k^{\frac{1}{3}}Z(s))
\]
where $\mathbb{Z}_+$ is the set of non-negative integers, $t_{-} = \min \{ t, 0 \}$, $b_{p,q,r,\ell}$ are $2P$-periodic complex-valued $C^{\infty}$ functions, $Z$ is a
$2P$-periodic real-valued $C^{\infty}$ function that is positive on the illuminated region $\partial K^{\rm IL}$,
negative on the shadow region $\partial K^{\rm SR}$, and vanishes precisely to first order at the shadow
boundary $\partial K^{\rm SR}$. Finally $\Psi^{r,\ell}$ are complex-valued $C^{\infty}$ functions which admit
the asymptotic expansions 
\[
	\Psi^{r,\ell}(\tau) \sim \sum_{j \in \mathbb{Z}_+} a_{r,\ell,j} \tau^{1+\ell-2r-3j}
	\qquad
	\text{as } \tau \to +\infty
\]
and rapidly decrease in the sense of Schwarz as $\tau \to -\infty$.
\end{theorem}
For concise definitions of H\"{o}rmander classes and asymptotic expansions we refer to \cite[\S2.2]{EcevitReitich09}.

In this section, we study the asymptotic behavior of the terms $a_{p,q,r,\ell}$ appearing in the expansion \eqref{eq:MT85Neumann},
and first show that they belong to the H\"{o}rmander class $S^{\vartheta(p,q,r,\ell)}_{\frac{2}{3},\frac{1}{3}} ([0,2P]\times(0,\infty))$ where
\[
	\vartheta(p,q,r,\ell) = -\frac{1+2p+3q+r+\ell}{3} + (\ell+1)_{-} +
	\left\{
		\begin{array}{ll}
			0, & 1+\ell-2r-p < 0, 
			\\ [0.3 em]
			\frac{1+\ell-2r-p}{3}, & 1+\ell-2r-p \ge 0.
		\end{array}
	\right.
\]
We then use this result to carry out a similar study for the expressions $\rho^{\rm slow}_\beta$ which we now define.
Throughout the text, we use the standard convention that an empty sum is zero.

\begin{definition} \label{def:quantities}
Given $\beta \in \mathbb{Z}_{+}$, we define
\begin{equation} \label{eq:sigmarhobetaslow}
	\sigma^{\rm slow}_{\beta} = \sum_{(p,q,r,\ell) \in \mathcal{F}_{\beta}} a_{p,q,r,\ell}
	\qquad
	\text{and}
	\qquad
	\rho^{\rm slow}_{\beta} = \eta^{\rm slow} - \sigma^{\rm slow}_{\beta}
\end{equation}
where
\[
	\mathcal{F}_{\beta}
	= \left\{ (p,q,r,\ell) \in \mathbb{Z}_{+} \times \mathbb{Z}_{+} \times \mathbb{Z}_{+} \times (-\mathbb{N}) : \vartheta(p,q,r,\ell) > - \frac{\beta}{3} \right\}.
\]
We also set

\begin{equation} \label{eq:sigmarhobeta}
	\sigma_\beta = e^{ik \, \alpha \cdot \gamma} \sigma^{\rm slow}_{\beta}
	\qquad
	\text{and}
	\qquad
	\rho_{\beta} = \eta - \sigma_{\beta} = e^{ik \, \alpha \cdot \gamma} \rho^{\rm slow}_{\beta}.
\end{equation}
\end{definition}

As mentioned in the introduction, the use of C\'{e}a's lemma gives rise to a factor depending on the wavenumber which effectively means the need for higher number of degrees
freedom with increasing $k$. The goal of the previous definition is to eliminate this factor by incorporating sufficiently many terms in the asymptotic expansion into
any given continuous and coercive integral equation formulations. Indeed, as shown in the following analysis, the definition of $\sigma^{\rm slow}_\beta$
\eqref{eq:sigmarhobetaslow} is optimal in the sense that it leads to the balancing factor $k^{-\beta/3}$ subject to a minimum number of terms $a_{p,q,r,\ell}$
incorporated into the integral equation. This represents a very important step in the design of Galerkin approximation spaces which can provide prescribed error
tolerances with the utilization of frequency independent numbers of degrees of freedom. 

Let us begin our analysis by giving the following result which is immediate from the asymptotic behavior of $\Psi^{r,\ell}$ described in Theorem \ref{thm:asymp-exp}.
In what follows, the symbol $\lesssim$ is used to mean that an inequality holds up to a wavenumber independent factor. 

\begin{lemma} \label{lemma:Psilr}
For all $p,q,r \in \mathbb{Z}_+$ and $\ell \in \mathbb{Z}$, the estimates
\begin{equation} \label{eq:psi-derivatives}
	|(\Psi^{r,\ell})^{(p)}(\tau)|
	\lesssim
	\left\{ \!\!
		\begin{array}{ll}
			(1 + |\tau|)^{\gamma_{r,\ell}-p},
			& \text{if } p > 1+\ell -2r \ge 0,
			\\
			(1+|\tau|)^{1+\ell-2r-p},
			& \text{otherwise},
		\end{array}
	\right.
\end{equation}
hold for all $\tau \in \mathbb{R}$ where
\[
	1+\ell -2r \equiv \gamma_{r,\ell} \mod 3
	\qquad
	\text{with}
	\qquad
	\gamma_{r,\ell} \in \{ -3,-2,-1\}.
\]
\end{lemma}

With the aid of Theorem~\ref{thm:asymp-exp}, Lemma \ref{lemma:Psilr} and Lemma \ref{lemma:generalderivatives} in Appendix~\ref{sec:auxiliary},
we now characterize the H\"{o}rmander classes of
$a_{p,q,r,\ell}$ and derive wavenumber explicit estimates on their derivatives.

\begin{lemma}[H\"{o}rmander classes of $a_{p,q,r,\ell}$]
\label{lemma:apqrl-Hor}
For any $p,q,r \in \Zplus$ and $\ell \in \mathbb{Z}$, $a_{p,q,r,l}$ belongs to the H\"{o}rmander class $S^{\vartheta(p,q,r,\ell)}_{\frac{2}{3},\frac{1}{3}} ([0,2P]\times(0,\infty))$.
\end{lemma}
\begin{proof}
Given $n,m,p,q,r \in \Zplus$ and $\ell \in \mathbb{Z}$, an appeal to Lemma~\ref{lemma:generalderivatives} entails
\begin{equation} \label{eq:apq-derivatives}
	|D_s^{n}D_{k}^{m} \, a_{p,q,r,\ell}(s,k)|
	\lesssim
	k^{-\frac{1+2p+3q+r+\ell}{3}+(\ell+1)_{-}-m}
	\sum_{0 \le j \le n+m}
	k^{\frac{j}{3}} 
	|(\Psi^{r,\ell})^{(p+j)}(k^{\frac{1}{3}}Z(s))|
\end{equation}
for all $(s,k) \in [0,2P] \times (0,\infty)$.

When $1+\ell-2r <0$, use of \eqref{eq:psi-derivatives} in \eqref{eq:apq-derivatives} implies
\begin{align*}
	|D_s^{n}D_{k}^{m} \, a_{p,q,r,\ell}(s,k)|
	& \lesssim
	k^{-\frac{1+2p+3q+r+\ell}{3}+(\ell+1)_{-}-m}
	\sum_{0 \le j \le n+m}
	k^{\frac{j}{3}} (1+k^{\frac{1}{3}} |Z(s)|)^{1+\ell-2r-p-j}
	\\
	&
	\lesssim
	k^{-\frac{1+2p+3q+r+\ell}{3}+(\ell+1)_{-}-m}
	\sum_{0 \le j \le n+m}
	k^{\frac{j}{3}}
	\lesssim
	k^{-\frac{1+2p+3q+r+\ell}{3}+(\ell+1)_{-}+\frac{n}{3}-\frac{2m}{3}}
\end{align*}
and this, in turn, implies that $a_{p,q,r,l} \in S^{\vartheta(p,q,r,\ell)}_{\frac{2}{3},\frac{1}{3}} ([0,2P]\times(0,\infty))$.

When $1+\ell-2r \ge 0$, use of \eqref{eq:psi-derivatives} in \eqref{eq:apq-derivatives} gives
\begin{multline} \label{eq:detail}
	|D_s^{n} D_{k}^{m} \, a_{p,q,r,\ell}(s,k)|
	\lesssim
	k^{-\frac{1+2p+3q+r+\ell}{3}+(\ell+1)_{-}-m}
	\\
	\times
	\Big\{
		\sum_{0 \le j \le n+m \atop j \le 1+\ell-2r-p} k^{\frac{j}{3}} (1+k^{\frac{1}{3}} |Z(s)|)^{1+\ell-2r-p-j}
		+ \sum_{0 \le j \le n+m \atop j > 1+\ell-2r-p} k^{\frac{j}{3}} (1+k^{\frac{1}{3}} |Z(s)|)^{\gamma_{r,\ell}-p-j}
	\Big\}.
\end{multline}
If $1+\ell-2r-p < 0$, \eqref{eq:detail} reduces to
\begin{align*}
	|D_s^{n}D_{k}^{m} \, a_{p,q,r,\ell}(s,k)|
	& \lesssim
	k^{-\frac{1+2p+3q+r+\ell}{3}+(\ell+1)_{-}-m}
	\sum_{0 \le j \le n+m} k^{\frac{j}{3}} (1+k^{\frac{1}{3}} |Z(s)|)^{\gamma_{r,\ell}-p-j}
	\\
	& \lesssim
	k^{-\frac{1+2p+3q+r+\ell}{3}+(\ell+1)_{-}-m}
	\sum_{0 \le j \le n+m} k^{\frac{j}{3}}
	\lesssim
	k^{-\frac{1+2p+3q+r+\ell}{3}+(\ell+1)_{-}-m} \,
	k^{\frac{n+m}{3}}
	\\
	& \lesssim
	k^{-\frac{1+2p+3q+r+\ell}{3}+(\ell+1)_{-}+\frac{n}{3}-\frac{2m}{3}}
	\lesssim k^{\vartheta(p,q,r,\ell)+\frac{n}{3}-\frac{2m}{3}}
\end{align*}
so that $a_{p,q,r,l} \in S^{\vartheta(p,q,r,\ell)}_{\frac{2}{3},\frac{1}{3}} ([0,2P]\times(0,\infty))$. If $1+\ell-2r-p \ge 0$, 
setting $J = \min \{ n+m, 1+\ell-2r - p\}$, \eqref{eq:detail} takes on the form
\begin{align*}
	|D_s^{n}D_{k}^{m} \, a_{p,q,r,\ell}(s,k)|
	& \lesssim
	k^{-\frac{1+2p+3q+r+\ell}{3}+(\ell+1)_{-}-m}
	\\
	& \quad \times
	\Big\{
		\sum_{0 \le j  \le J} k^{\frac{j}{3}} (1+k^{\frac{1}{3}} |Z(s)|)^{1+\ell-2r-p-j}
		+ \sum_{J < j \le n+m} k^{\frac{j}{3}} (1+k^{\frac{1}{3}} |Z(s)|)^{\gamma_{r,\ell}-p-j}
	\Big\}
	\\
	& \lesssim
	k^{-\frac{1+2p+3q+r+\ell}{3}+(\ell+1)_{-}-m}
	\times
	\Big\{
		\sum_{0 \le j \le J} k^{\frac{j}{3}} k^{\frac{1+\ell-2r-p-j}{3}}
		+ \sum_{J < j \le n+m} k^{\frac{j}{3}}
	\Big\}
	\\
	& \lesssim
	k^{-\frac{1+2p+3q+r+\ell}{3}+(\ell+1)_{-}-m}
	\times
	\Big\{
		k^{\frac{1+\ell-2r-p}{3}}
		+ \sum_{J < j \le n+m} k^{\frac{j}{3}}
	\Big\}.
\end{align*}
This entails, when $0 \le 1+\ell-2r - p < n+m$
\begin{align*}
	|D_s^{n}D_{k}^{m} \, a_{p,q,r,\ell}(s,k)|
	& \lesssim
	k^{-\frac{1+2p+3q+r+\ell}{3}+(\ell+1)_{-}-m}
	\times
	\left\{
		k^{\frac{1+\ell-2r-p}{3}}
		+ k^{\frac{n+m}{3}}
	\right\}
	\\
	&
	\lesssim
	k^{-\frac{1+2p+3q+r+\ell}{3}+(\ell+1)_{-}-m} \, k^{\frac{n+m}{3}}
	\lesssim k^{-\frac{1+2p+3q+r+\ell}{3}+(\ell+1)_{-}+\frac{n}{3}-\frac{2m}{3}}
	\\
	& \lesssim k^{\vartheta(p,q,r,\ell)+\frac{n}{3}-\frac{2m}{3}},
\end{align*}
and when $n+m \le 1+\ell-2r - p$
\begin{align*}
	|D_s^{n}D_{k}^{m} \, a_{p,q,r,\ell}(s,k)|
	& \lesssim
	k^{-\frac{1+2p+3q+r+\ell}{3} +(\ell+1)_{-}-m}
	k^{\frac{1+\ell-2r-p}{3}}
	\\
	& \lesssim
	k^{\vartheta(p,q,r,\ell)-m}
	\lesssim k^{\vartheta(p,q,r,\ell)+\frac{n}{3}-\frac{2m}{3}}.
\end{align*}
These show that $a_{p,q,r,l} \in S^{\vartheta(p,q,r,\ell)}_{\frac{2}{3},\frac{1}{3}} ([0,2P]\times(0,\infty))$ when $1+\ell-2r-p \ge 0$.
\end{proof}

For the developments that follow, we define
\[
	W(s,k) = k^{-\frac{1}{3}} + |\omega(s)|
	\quad
	\text{with}
	\quad
	\omega(s) = (s-t_1)(t_2-s)
\]
for any $k > 0$.

\begin{lemma}[Wavenumber explicit estimates on the derivatives of $a_{p,q,r,\ell}$]
\label{lemma:apqrl-der}
Given $k_0 > 0$ and $n,p,q,r \in \Zplus$ and $\ell \in \mathbb{Z}$, the estimate
\begin{equation}
\label{eq:apq-estimate}
	| D_s^{n} a_{p,q,r,\ell}(s,k)|
	\lesssim
	k^{-\frac{1+2p+3q+r+\ell}{3}+(\ell+1)_{-}}
	\times
	\left\{
		\begin{array}{ll}
			\sum\limits_{0 \le j \le n} W(s,k)^{-j},
			& 1+\ell-2r-p < 0,
			\\
			\lambda_{p,r,\ell,n}(s,k),
			& 1+\ell-2r-p \ge 0,
		\end{array}
	\right.
\end{equation}
holds for all $(s,k) \in [0,2P] \times [k_0,\infty)$ where
\[
	\lambda_{p,r,\ell,n}(s,k)
	= k^{\frac{1+\ell-2r-p}{3}}
	+ \sum_{1+\ell-2r-p < j \le n} W(s,k)^{-j}.
\]

\end{lemma}
\begin{proof}
Given $n,p,q,r \in \Zplus$ and $\ell \in \mathbb{Z}$, use of Lemma~\ref{lemma:generalderivatives} entails
\begin{equation} \label{eq:apq-derivativess}
	|D_s^{n} \, a_{p,q,r,\ell}(s,k)|
	\lesssim
	k^{-\frac{1+2p+3q+r+\ell}{3} +(\ell+1)_{-}}
	\sum_{0 \le j \le n}
	k^{\frac{j}{3}} 
	|(\Psi^{r,\ell})^{(p+j)}(k^{\frac{1}{3}}Z(s))|
\end{equation}
for all $(s,k) \in [0,2P] \times (0,\infty)$.

When $1+\ell-2r <0$, use of \eqref{eq:psi-derivatives} in \eqref{eq:apq-derivativess} implies
\begin{align*}
	|D_s^{n} \, a_{p,q,r,\ell}(s,k)|
	& \lesssim
	k^{-\frac{1+2p+3q+r+\ell}{3} +(\ell+1)_{-}}
	\sum_{0 \le j \le n}
	k^{\frac{j}{3}} (1+k^{\frac{1}{3}} |Z(s)|)^{1+\ell-2r-p-j}
	\\
	& \lesssim
	k^{-\frac{1+2p+3q+r+\ell}{3} +(\ell+1)_{-}}
	\sum_{0 \le j \le n}
	k^{\frac{j}{3}} (1+k^{\frac{1}{3}} |\omega(s)|)^{1+\ell-2r-p-j}
	\\
	& \lesssim
	k^{-\frac{1+2p+3q+r+\ell}{3} +(\ell+1)_{-}}
	\sum_{0 \le j \le n}
	k^{\frac{j}{3}} (1+k^{\frac{1}{3}} |\omega(s)|)^{-j}
	\\
	& \lesssim k^{-\frac{1+2p+3q+r+\ell}{3} +(\ell+1)_{-}}
	\sum_{0 \le j \le n}
	(k^{-\frac{1}{3}} + |\omega(s)|)^{-j}.
\end{align*}

When $1+\ell-2r \ge 0$, use of \eqref{eq:psi-derivatives} in \eqref{eq:apq-derivativess} gives
\begin{multline} \label{eq:detail1}
	|D_s^{n} \, a_{p,q,r,\ell}(s,k)|
	\lesssim
	k^{-\frac{1+2p+3q+r+\ell}{3} +(\ell+1)_{-}}
	\\
	\times
	\Big\{
		\sum_{0 \le j \le n \atop j \le 1+\ell-2r-p} k^{\frac{j}{3}} (1+k^{\frac{1}{3}} |Z(s)|)^{1+\ell-2r-p-j}
		+ \sum_{0 \le j \le n \atop j > 1+\ell-2r-p} k^{\frac{j}{3}} (1+k^{\frac{1}{3}} |Z(s)|)^{\gamma_{r,\ell}-p-j}
	\Big\}.
\end{multline}
If $1+\ell-2r-p < 0$, \eqref{eq:detail1} reduces to
\begin{align*}
	|D_s^{n} \, a_{p,q,r,\ell}(s,k)|
	& \lesssim
	k^{-\frac{1+2p+3q+r+\ell}{3} +(\ell+1)_{-}}
	\sum_{0 \le j \le n} k^{\frac{j}{3}} (1+k^{\frac{1}{3}} |Z(s)|)^{\gamma_{r,\ell}-p-j}
	\\
	& \lesssim
	k^{-\frac{1+2p+3q+r+\ell}{3} +(\ell+1)_{-}}
	\sum_{0 \le j \le n} k^{\frac{j}{3}} (1+k^{\frac{1}{3}} |Z(s)|)^{-j}
	\\
	& \lesssim
	k^{-\frac{1+2p+3q+r+\ell}{3} +(\ell+1)_{-}}
	\sum_{0 \le j \le n}  k^{\frac{j}{3}} (1+k^{\frac{1}{3}} |\omega(s)|)^{-j}
	\\
	& =
	k^{-\frac{1+2p+3q+r+\ell}{3} +(\ell+1)_{-}}
	\sum_{0 \le j \le n}  (k^{-\frac{1}{3}}+|\omega(s)|)^{-j}.
\end{align*}
If $1+\ell-2r-p \ge 0$, 
setting $J = \min \{ n, 1+\ell-2r - p\}$, \eqref{eq:detail1} becomes
\begin{align*}
	|D_s^{n}D_{k}^{m} \, a_{p,q,r,\ell}(s,k)|
	& \lesssim
	k^{-\frac{1+2p+3q+r+\ell}{3} +(\ell+1)_{-}}
	\\
	& \quad \times
	\Big\{
		\sum_{0 \le j  \le J} k^{\frac{j}{3}} (1+k^{\frac{1}{3}} |Z(s)|)^{1+\ell-2r-p-j}
		+ \sum_{J < j \le n} k^{\frac{j}{3}} (1+k^{\frac{1}{3}} |Z(s)|)^{\gamma_{r,\ell}-p-j}
	\Big\}
	\\
	& \lesssim
	k^{-\frac{1+2p+3q+r+\ell}{3} +(\ell+1)_{-}}
	\\
	& \quad \times
	\Big\{
		\sum_{0 \le j \le J} k^{\frac{j}{3}} k^{\frac{1+\ell-2r-p-j}{3}}
		+ \sum_{J < j \le n} k^{\frac{j}{3}} (1+k^{\frac{1}{3}} |\omega(s)|)^{\gamma_{r,\ell}-p-j}
	\Big\}
	\\
	& \lesssim
	k^{-\frac{1+2p+3q+r+\ell}{3} +(\ell+1)_{-}}
	\times
	\Big\{
		k^{\frac{1+\ell-2r-p}{3}}
		+ \sum_{J < j \le n} k^{\frac{j}{3}} (1+k^{\frac{1}{3}} |\omega(s)|)^{-j}
	\Big\}
	\\
	& =
	k^{-\frac{1+2p+3q+r+\ell}{3} +(\ell+1)_{-}}
	\times
	\Big\{
		k^{\frac{1+\ell-2r-p}{3}}
		+ \sum_{j=J+1}^{n} (k^{-\frac{1}{3}} + |\omega(s)|)^{-j}
	\Big\}.
\end{align*}
Thus the result follows.
\end{proof}

Using Lemma~\ref{lemma:apqrl-der}, we can obtain the following.

\begin{corollary}[Wavenumber explicit estimates on the derivatives of $a_{p,q,r,\ell}$] \label{cor:apkldersimple}
Given $k_0 > 0$ and $n,p,q,r \in \Zplus$ and $\ell \in \mathbb{Z}$, the estimate
\begin{equation}
\label{eq:apq-estimatecompact}
	| D_s^{n} a_{p,q,r,\ell}(s,k)|
	\lesssim
	k^{\vartheta(p,q,r,\ell)} \, W(s,k)^{-n} 
\end{equation}
holds for all $(s,k) \in [0,2P] \times [k_0,\infty)$.
\end{corollary}

We now make use of Lemma \ref{lemma:apqrl-Hor} and Corollary \ref{cor:apkldersimple} to characterize the
H\"{o}rmander class and derive wavenumber explicit estimates on the derivatives of the envelope $\rho_\beta^{\rm slow}$
introduced in Definition \ref{def:quantities}.

\begin{theorem} \label{thm:etamder}
For any $\beta \in \mathbb{Z}_+$, the envelope $\rho_\beta^{\rm slow}(s,k)$ belongs to $S^{-\frac{\beta}{3}}_{\frac{2}{3},\frac{1}{3}}([0,P] \times (0,\infty))$.
Moreover, given $k_0 > 1$ and $n \in \Zplus$, the estimate
\[
	|D_s^n \rho_\beta^{\rm slow}(s,k)|
	\lesssim k^{-\frac{\beta}{3}} \, W(s,k)^{-n}
\]
holds for all $(s,k) \in [0,2P] \times [k_0,\infty)$.
\end{theorem}
\begin{proof}
Since $\eta^{\rm slow} \in S^{0}_{\frac{2}{3},\frac{1}{3}}$ and $a_{p,q,r,\ell} \in S^{\vartheta(p,q,r,\ell)}_{\frac{2}{3},\frac{1}{3}}$,
given $\beta \in \Zplus$, by definition of H\"{o}rmander classes, we have
$\rho^{\rm slow}_{\beta}
= \eta^{\rm slow} - \sigma^{\rm slow}_{\beta}
= \eta^{\rm slow} - \sum_{(p,q,r,\ell) \in \mathcal{F}_\beta} a_{p,q,r,\ell}
\in S^{\vartheta_\beta}_{\frac{2}{3},\frac{1}{3}}$ with
\[
	\vartheta_{\beta}
	= \max \{\vartheta(p,q,r,\ell):(p,q,r,\ell) \in \Zplus \times \Zplus \times \Zplus \times (-\mathbb{N}) \backslash \mathcal{F}_{\beta}\}.
\]
By definitions of $\vartheta(p,q,r,\ell)$ and $\mathcal{F}_{\beta}$, we have $\vartheta_{\beta} = -\frac{\beta}{3}$ so that
$\rho^{\rm slow}_{\beta} \in S^{-\frac{\beta}{3}}_{\frac{2}{3},\frac{1}{3}}$. 

As for the estimate, given $n \in \Zplus$,
since
\[
	\rho_{\beta}^{\rm slow} = \rho_{\beta+n}^{\rm slow}
	+ \sum_{(p,q,r,\ell) \in \mathcal{F}_{\beta+n} \backslash \mathcal{F}_{\beta}} a_{p,q,r,\ell}
\]
and $\rho_{\beta+n}^{\rm slow} \in S^{-\frac{\beta+n}{3}}_{\frac{2}{3},\frac{1}{3}}$, we have
\[
	|D_s^n \rho_{\beta}^{\rm slow}(s,k)|
	\lesssim
	(1+k)^{-\frac{\beta+n}{3} + \frac{n}{3}} +
	\sum_{(p,q,r,\ell) \in \mathcal{F}_{\beta+n} \backslash \mathcal{F}_{\beta}}
	|D_s^n a_{p,q,r,\ell}(s,k)|.
\]
Thus the preceding corollary implies
\begin{align*}
	|D_s^n \rho_{\beta}^{\rm slow}(s,k)|
	& \lesssim
	(1+k)^{-\frac{\beta+n}{3} + \frac{n}{3}}
	+  k^{\max \{ \nu(p,q,r,\ell) : (p,q,r,\ell) \in \mathcal{F}_{\beta+n} \backslash \mathcal{F}_{\beta} \}} \, W(s,k)^{-n} 
	\\
	& \lesssim
	(1+k)^{-\frac{\beta+n}{3} + \frac{n}{3}} + 
	k^{-\frac{\beta}{3}} \, W(s,k)^{-n} 
	\lesssim
	k^{-\frac{\beta}{3}} \, W(s,k)^{-n},
\end{align*}
and this completes the proof.
\end{proof}

\section{Galerkin boundary element methods and convergence analyses} \label{sec:4}

Throughout this section we assume that an integral equation formulation
\begin{equation} \label{eq:inteq}
	\mathcal{R}_k \eta = f
\end{equation}
is given to deal with the problem \eqref{eq:Helmholtz} and is continuous and coercive for all $k \ge k_0$ for some $k_0 > 1$, with continuity and coercivity
constants $C_k$ and $c_k$. We also assume that $\sigma^{\rm slow}_{\beta}$ \eqref{eq:sigmarhobetaslow} is available for some $\beta \in \Zplus$.
In this case,
\eqref{eq:inteq} can be re-written in terms of the new unknown $\rho_\beta = \eta - \sigma_\beta$ as
\begin{equation} \label{eq:inteqnew}
	\mathcal{R}_k \rho_\beta = f_\beta
\end{equation}
where $f_\beta = f - \mathcal{R}_k\sigma_\beta$. Note that $\rho_\beta = \eta$ and $f_\beta = f$ when $\beta = 0$.

\begin{definition} \label{def:Galsol}
We define the \emph{$\beta$-asymptotic Galerkin approximation $\hat{\eta}_\beta$ to $\eta$} associated with a finite dimensional
subspace $\mathcal{G}$ of $L^2(\partial K)$ as
\begin{equation} \label{eq:hatrhobeta}
	\hat{\eta}_\beta = \sigma_\beta + \hat{\rho}_\beta \in \sigma_\beta + \mathcal{G}	
\end{equation}
where
$\hat{\rho}_\beta \in \mathcal{G}$ is the unique solution to the Galerkin formulation 
\begin{equation} \label{eq:Galfor}
	\langle \hat{\mu}, \mathcal{R}_k \hat{\rho}_\beta \rangle = \langle \hat{\mu}, f_\beta \rangle,
	\qquad
	\text{for all }
	\hat{\mu} \in \mathcal{G},
\end{equation}
of the integral equation \eqref{eq:inteqnew}.
\end{definition}

In virtue of \eqref{eq:hatrhobeta} and Definition \ref{def:quantities}, we observe that $\eta - \hat{\eta}_\beta = \rho_\beta - \hat{\rho}_\beta$. 
Accordingly, the Galerkin approximation spaces defined in the form
\[
	\mathcal{G} = e^{ik \, \alpha \cdot \gamma} \mathcal{G}^{\rm slow}
\]
capture the oscillations in $\rho_\beta = e^{ik \, \alpha \cdot \gamma} \rho_\beta^{\rm slow}$ exactly. This, in turn, reduces the problem to
the design of approximation spaces $\mathcal{G}^{\rm slow}$ so as to effectively resolve the boundary layers of $\rho_\beta^{\rm slow}$, 
as implied by Theorem~\ref{thm:etamder}, around the shadow boundaries with increasing $k$. In \S\ref{sec:method1} and \S\ref{sec:method2},
we introduce two different Galerkin approximation spaces that are designed to effectively resolve the aforementioned boundary layers, and where
their convergence analyses are also presented.
In particular, these analyses reveal that the explicit
knowledge of $\sigma_\beta$ implies that, provided the stability constant $C_k/c_k$ grows like $k^{\varrho}$ as $k \to \infty$ for some $\varrho>0$,
then it can be controlled by $k^{-\beta/3}$ choosing $\beta > 3\varrho$.

The design of \emph{frequency-adapted $\beta$-asymptotic Galerkin approximation space} in \S\ref{sec:method1}
and the \emph{$\beta$-asymptotic Galerkin approximation space based on frequency dependent changes of variables} in \S\ref{sec:method2},
replicate those proposed for solution of the corresponding Dirichlet problem in \cite{EcevitOzen17} and \cite{EcevitEruslu19} respectively.
However, the convergence analyses have non-trivial technicalities due to the differences between the wavenumber dependent estimates on the
derivatives of the densities (total field for the Neumann problem and normal derivative of total field for the Dirichlet problem). For the sake of
presentation, we refer to \cite{EcevitOzen17,EcevitEruslu19} for additional details on the proofs if needed.

\subsection{Frequency-adapted $\beta$-asymptotic Galerkin boundary element method}
\label{sec:method1}

For the construction of $\beta$-asymptotic frequency-adapted Galerkin approximation spaces,
given $k \ge 1$, a natural number $m$, real numbers $\varepsilon_1,\ldots,\varepsilon_m$ with
$0 < \varepsilon_m < \varepsilon_{m-1} < \cdots < \varepsilon_1 < \frac{1}{3}$, and
positive real numbers $\xi_1, \xi_2, \zeta_1,\zeta_2 $ satisfying
$t_1-\xi_1 < t_2-\xi_2$ and $t_2 + \zeta_2 < 2P+ t_1-\zeta_1$, the
\emph{illuminated region} ($IL$),
\emph{illuminated transitions} ($IT_1$ and $IT_2$),
\emph{shadow transitions} ($ST_1$ and $ST_2$),
\emph{shadow boundaries} ($SB_1$ and $SB_2$),
and \emph{shadow region} ($SR$),
in the parameter domain are defined as
\begin{align*}
	IL & = [t_1 + \xi_1k^{-\frac{1}{3} +\epsilon_1}, t_2 - \xi_2k^{-\frac{1}{3}+\epsilon_1}],
	\\
	IT_1 & = [t_1 + \xi_1 k^{-\frac{1}{3}+\epsilon_{m}}, t_1 + \xi_1 k^{-\frac{1}{3}+\epsilon_1}],
	\\
	IT_2 & = [t_2 - \xi_2 k^{-\frac{1}{3}+\epsilon_{1}}, t_2 - \xi_2 k^{-\frac{1}{3}+\epsilon_{m}}],
	\\
	SB_1 & = [t_1 - \zeta_1 k^{-\frac{1}{3} +\epsilon_m}, t_1 + \xi_1 k^{-\frac{1}{3}+\epsilon_m}],
	\\
	SB_2 & = [t_2 - \xi_1 k^{-\frac{1}{3} +\epsilon_m}, t_2 + \zeta_2 k^{-\frac{1}{3}+\epsilon_m}],
	\\
	ST_1 & = [t_1 - \zeta_1 k^{-\frac{1}{3}+\epsilon_1} , t_1 - \zeta_1 k^{-\frac{1}{3}+\epsilon_{m}}],
	\\
	ST_2 & = [t_2 +  \zeta_2 k^{-\frac{1}{3}+\epsilon_{m}} , t_2 + \zeta_2 k^{-\frac{1}{3}+\epsilon_{1}}],
	\\
	SR & = [t_2 + \zeta_2 k^{-\frac{1}{3}+\epsilon_1}, 2P+t_1-\zeta_1 k^{-\frac{1}{3}+\epsilon_1}].
\end{align*}
Note that as $k \to \infty$ the illuminated and shadow regions cover the entire boundary in
the parameter domain, and the remaining regions collapse to the shadow boundaries.
In order to resolve the singularities of $\rho_\beta^{\rm slow}$ in vicinities of shadow boundaries
as implied by the wavenumber explicit derivative estimates in Theorem~\ref{thm:etamder},
for $m > 1$, we partition each one of the four transition regions into $m-1$ subregions as
\begin{align*}
	IT_1^j & = [t_1 + \xi_1 k^{-\frac{1}{3}+\epsilon_{j+1}}, t_1 + \xi_1 k^{-\frac{1}{3}+\epsilon_j}],
	\\
	IT_2^j & = [t_2 - \xi_2 k^{-\frac{1}{3}+\epsilon_{j}}, t_2 - \xi_2 k^{-\frac{1}{3}+\epsilon_{j+1}}],
	\\
	ST_1^j & = [t_1 - \zeta_1 k^{-\frac{1}{3}+\epsilon_j} , t_1 - \zeta_1 k^{-\frac{1}{3}+\epsilon_{j+1}}],
	\\
	ST_2^j & = [t_2 +  \zeta_2 k^{-\frac{1}{3}+\epsilon_{j+1}} , t_2 + \zeta_2 k^{-\frac{1}{3}+\epsilon_{j}}],
\end{align*}
for  $j=1,\ldots , m-1$. These result in a total of $4m$ regions 
\[
	R_j = [a_j,b_j]
	= \left\{
		\begin{array}{ll}
			IT_1^j, & j = 1,\ldots,m-1,
			\\
			IT_2^{j-m}, & j = m+1,\ldots,2m-1,
			\\
			ST_1^{j-2m}, & j = 2m+1,\ldots,3m-1,
			\\
			ST_2^{j-3m}, & j = 3m+1,\ldots,4m-1,
		\end{array}
	\right.
\]
and
\[
	R_m = IL, \quad
	R_{2m} = SR, \quad
	R_{3m} = SB_1, \quad
	R_{4m} = SB_2,
\]
with the transition regions being redundant when $m=1$.
Identifying the spaces $L^{2} \left( \partial K \right)$ and $L^2 ( \cup_{j=1}^{4m} R_j)$ through the
parameterization $\gamma$, we now define the Galerkin approximation spaces, and the associated asymptotic Galerkin solutions.

\begin{definition}
For $m \in \mathbb{N}$ and $\mathbf{d} = \left(d_1,\ldots,d_{4m} \right) \in \mathbb{Z}_+^{4m}$, the $4m+|\mathbf{d}|$
dimensional \emph{frequency-adapted Galerkin approximation space} in $L^2 ( \partial K)$ is defined as the direct sum
\begin{equation} \label{eq:GalerkinPoly}
	\mathcal{G}_\mathbf{d}
	= e^{i k \, \alpha \cdot \gamma} \,
	\mathcal{G}_\mathbf{d}^{\rm slow}
	= e^{i k \, \alpha \cdot \gamma}
	\bigoplus_{j=1}^{4m} \,
	\charfunc_{R_j} \,
	\Pol_{d_j}
\end{equation}
where $\charfunc_{R}$ is the characteristic function, and $\Pol_{d}$ is the space of polynomials of degree at most $d$.
\end{definition}

\begin{definition}
For $m \in \mathbb{N}$ and $\mathbf{d} = \left(d_1,\ldots,d_{4m} \right) \in \mathbb{Z}_+^{4m}$, the \emph{$\beta$-asymptotic
frequency-adapted Galerkin approximation} $\hat{\eta}_\beta$ to $\eta$ is defined as
\[
	\hat{\eta}_\beta = \sigma_\beta + \hat{\rho}_\beta \in \sigma_\beta + \mathcal{G}_{\mathbf{d}}
\]
where $\sigma_{\beta}$ is as given in \eqref{eq:sigmarhobetaslow}, and
$\hat{\rho}_\beta = e^{i k \, \alpha \cdot \gamma} \hat{\rho}_\beta^{\rm slow} \in \mathcal{G}_\mathbf{d}$ is the unique solution of the Galerkin formulation 
\[
	\langle \hat{\mu}, \mathcal{R}_k \hat{\rho}_\beta \rangle = \langle \hat{\mu}, f_\beta \rangle,
	\qquad
	\text{for all }
	\hat{\mu} \in \mathcal{G}_{\mathbf{d}},
\]
of the integral equation \eqref{eq:inteqnew}.
\end{definition}

The approximation properties of the $\beta$-asymptotic frequency-adapted Galerkin method are given in the following.

\begin{theorem} \label{thm:ecevitozen}
Given $m \in \mathbb{N}$ and $(n_1,\ldots,n_{4m}) \in \mathbb{Z}_+^{4m}$, for $k \ge k_0$ and all
$\mathbf{d} = (d_1,\ldots,d_{4m}) \in \mathbb{N}^{4m}$ with $d_j \ge n_j-1$, we have
\begin{equation} \label{eq:algebraicestimate1}
	\Vert \eta - \hat{\eta}_\beta \Vert_{L^{2}(\partial K)}
	\lesssim
	\dfrac{C_k}{c_k} \, k^{-\frac{\beta}{3}} \,
	\sum_{j = 1}^{4m} \dfrac{1+ E(k,j)}{\left( d_{j} \right)^{n_{j}}}
\end{equation}
for the $\beta$-asymptotic frequency-adapted Galerkin approximation $\hat{\eta}_\beta \in \sigma_\beta + \mathcal{G}_{\mathbf{d}}$ to $\eta$. 
On the transition regions $($with $j' = j \mod m$ and $j' \in \{1,\ldots,m-1\}$$)$
\[
	E(k,j)
	= k^{-\frac{1-3\epsilon_{j'+1}}{6}} ( k^{\frac{\epsilon_{j'}-\epsilon_{j'+1}}{2}} )^{n_{j}},
	\qquad
	j \in \{1,\ldots,4m\} \backslash \{m,2m,3m,4m\},
\]
on the illuminated and shadow regions
\[
	E(k,j) = \delta_{n_j,1} \sqrt{\log k} + H[n_j-2] k^{-\frac{1-3\epsilon_1}{6}} ( k^{\frac{1-3\epsilon_1}{6}})^{n_{j}},
	\qquad
	j=m,2m,
\]
where $\delta$ and $H$ are the Kronecker delta and Heaviside functions,
and on the shadow boundaries
\[
	E(k,j) = k^{-\frac{1}{6}} \left( k^{\epsilon_{m}} \right)^{n_{j}},
	\qquad
	j=3m,4m.
\]
\end{theorem}

\begin{proof}
Writing $\hat{\rho}_\beta = e^{ik \, \alpha \cdot \gamma} \hat{\rho}_\beta^{\rm slow}$ for the unique solution of \eqref{eq:Galfor},
we have
\begin{equation}
	\eta - \hat{\eta}_\beta
	= (\sigma_\beta + \rho_\beta) - (\sigma_\beta + \hat{\rho}_\beta)
	= \rho_\beta - \hat{\rho}_\beta
	= e^{ik \, \alpha \cdot \gamma} \sum_{j=1}^{4m} \charfunc_{R_j} (\rho_\beta^{\rm slow} - \hat{\rho}_{\beta}^{\rm slow}).
\end{equation}
Accordingly, when $\mathcal{G} = \mathcal{G}_{\mathbf{d}}$, using C\'{e}a's lemma, we obtain
\begin{align*}
	\Vert \eta - \hat{\eta}_\beta \Vert
	& = \Vert \sum_{j=1}^{4m} \charfunc_{R_j} (\rho_\beta^{\rm slow} - \hat{\rho}_{\beta}^{\rm slow}) \Vert 
	\\
	& \le \frac{C_k}{c_k} \inf \{ \Vert \sum_{j=1}^{4m} \charfunc_{R_j} (\rho_\beta^{\rm slow} - p_j) \Vert :
	(p_1,\ldots,p_{4m}) \in \mathbb{P}_{d_1} \times \ldots \times \mathbb{P}_{d_{4m}} \}
	\\
	& \le \frac{C_k}{c_k} \sum_{j=1}^{4m} \inf_{p_j \in \mathbb{P}_{d_j}} \Vert \rho_\beta^{\rm slow} - p_j \Vert_{L^2([a_j,b_j])}.
\end{align*}
Therefore, by Theorem~\ref{thm:etamder} above and Theorem \ref{thm:pae} in Appendix~\ref{sec:auxiliary}, we have
\begin{align*}
	\Vert \eta - \hat{\eta}_\beta \Vert
	& \lesssim \frac{C_k}{c_k} k^{-\frac{\beta}{3}}
	\sum_{j=1}^{4m} \frac{\mathcal{W}(k;n_j;a_j,b_j)}{(d_j)^{n_j}}
\end{align*}
for all positive integers $d_j \ge n_j-1$ ($j = 1,\ldots,4m$) where
\begin{equation} \label{eq:mathcalW}
	 \mathcal{W}(k;n;a,b) = \left[ \int_a^b \frac{(s-a)^{n} \, (b-s)^{n}}{W(s,k)^{2n}} \, ds \right]^{\frac{1}{2}}.
\end{equation}
This inequality when combined with the next lemma gives the desired result.
\end{proof}

\begin{lemma} \label{lemma:babap}
For all $n \in \mathbb{Z}_+$ and all $k \ge 1$, we have:

\begin{itemize}
\item[(i)] Illuminated and shadow regions:
If $0 < \epsilon < \frac{1}{3}$, $a = t_{1} + \xi_{1} k^{-\frac{1}{3} + \epsilon}$ and $b = t_{2} - \xi_{2} k^{-\frac{1}{3} + \epsilon}$, or
$a = t_{2} + \zeta_{2} k^{-\frac{1}{3} + \epsilon}$ and $b = 2P+ t_{1} - \zeta_{1} k^{-\frac{1}{3}+\epsilon}$, then
\begin{equation} \label{eq:ISest}
	\mathcal{W}(k;n;a,b) 
	\lesssim 1 + \delta_{n,1} \sqrt{\log k} + H[n-2] k^{\frac{1 - 3 \epsilon}{6}(n-1)}.
\end{equation}
\item[(ii)] Illuminated and shadow transitions:
If $0 < \delta < \epsilon < \frac{1}{3}$,
$a = t_{1} + \xi_1 k^{-\frac{1}{3} + \delta}$ and $b = t_{1} + \xi_1 k^{-\frac{1}{3} + \epsilon}$,
or $a = t_{1} - \zeta_1 k^{-\frac{1}{3} + \epsilon}$ and $b = t_{1} - \zeta_1 k^{-\frac{1}{3} + \delta}$,
or $a = t_{2} + \zeta_2 k^{-\frac{1}{3} + \delta}$ and $b = t_{2} + \zeta_2 k^{-\frac{1}{3} + \epsilon}$,
or $a = t_{2} - \xi_2 k^{-\frac{1}{3} + \epsilon}$ and $b = t_{2} - \xi_2 k^{-\frac{1}{3} + \delta}$,
then
\begin{equation} \label{eq:ISTest}
	\mathcal{W}(k;n;a,b) 
	\lesssim 1 + k^{-\frac{1-3\delta}{6}} k^{\frac{\epsilon-\delta}{2} n}.
\end{equation}
\item[(iii)] Shadow boundaries:
If $0 \le \epsilon, \delta < \frac{1}{3}$, $a = t_{1} - \zeta_1 k^{-\frac{1}{3} + \delta}$ and $b = t_{1} + \xi_1 k^{-\frac{1}{3} + \epsilon}$,
or $a = t_{2} - \xi_1 k^{-\frac{1}{3} + \epsilon}$ and $b = t_{2} + \zeta_2 k^{-\frac{1}{3} + \delta}$,
then
\begin{equation} \label{eq:SBest}
	\mathcal{W}(k;n;a,b) 
	\lesssim 1 + k^{-\frac{1}{6}} k^{\frac{\epsilon+\delta}{2} n}.
\end{equation}
\end{itemize}
\end{lemma}

\begin{proof}
In each of the three cases, the analyses leading into the given estimates are similar for each of the given pairs of parameters $a$
and $b$, so we present the proof for only the very first pairs.

In any case, we have
\begin{equation} \label{eq:intest0}
	W(k;0;a,b) = (b-a)^{\frac{1}{2}} \lesssim 1,
\end{equation}
so we assume $n \ge 1$. When $1 \le k \le k_0$ for some $k_0 >1$, we clearly have
\begin{equation} \label{eq:intestk0}
	\mathcal{W}(k;n;a,b) \lesssim 1,
\end{equation}
and therefore we can assume that $k$ is sufficiently large. In this case, with $T = \frac{t_2 - t_1}{2}$, we have
\begin{equation} \label{eq:Wfactor}
	W(s,k) =
	\left\{
		\begin{array}{rl}
			   (s-c_{I}) \, (d_{I} - s), & s \in [t_{1},t_{2}],
			   \\
			   (c_{S}-s) \, (d_{S} - s), & s \in [0,2P] \backslash [t_{1},t_{2}],
		\end{array}
	\right.
\end{equation}
where
\begin{align*}
	c_{I} = P - \sqrt{T^{2} + k^{-\frac{1}{3}}},
	\ \,
	d_{I} = P + \sqrt{T^{2} + k^{-\frac{1}{3}}},
	\ \,
	c_{S} = P - \sqrt{T^{2} - k^{-\frac{1}{3}}},
	\ \,
	d_{S} = P + \sqrt{T^{2} - k^{-\frac{1}{3}}}.
\end{align*}

For (i) and (ii), we use \eqref{eq:Wfactor} in Lemma~\ref{lemma:exactintegral} in Appendix~\ref{sec:auxiliary} to obtain
\begin{equation} \label{eq:Wexact-i-ii}
	\mathcal{W}(k;n;a,b)^2
	= \sum_{\substack{0 \le p,q \le n \\ 1 \le j \le 2n}}
	\binom{4n-j-1}{2n-j} \binom{n}{p} \binom{n}{q}
	\dfrac{(-1)^{n} \, \mathcal{F}(a,b;a,b;c_{I},d_{I};n,p,q;j)}{(d_{I}-c_{I})^{4n-j}}
\end{equation}
where for $2n-(p+q+j) = -1$
\begin{align*}
	\mathcal{F}(a,b;a,b;c_{I},d_{I};n,p,q;j)
	= \left( c_I-a \right)^{p} \left( c_I-b \right)^{q} \log \left( \dfrac{b-c_I}{a-c_I} \right)
	+ \left( a-d_I \right)^{p} \left( b-d_I \right)^{q} \log \left( \dfrac{d_I-a}{d_I-b} \right),
\end{align*}
and for $2n-(p+q+j) \ne -1$
\begin{align*}
	\mathcal{F}(a,b;a,b;c_{I},d_{I};n,p,q;j)
	& = \dfrac{\left( c_I-a \right)^{p} \left( c_I-b \right)^{q}}{2n-(p+q+j)+1}
	\left[ \left( b-c_I \right)^{2n-(p+q+j)+1} - \left( a-c_I \right)^{2n-(p+q+j)+1} \right]
	\\
	& + \dfrac{\left( a-d_I \right)^{p} \left( b-d_I \right)^{q}}{2n-(p+q+j)+1}
	\left[ \left( d_I-a \right)^{2n-(p+q+j)+1} - \left( d_I-b \right)^{2n-(p+q+j)+1} \right].
\end{align*}

For (i), as $k \to \infty$, we have 
$a - c_I \asymp k^{-\frac{1}{3} + \epsilon}$,
$b-c_I \asymp 1$, 
$d_I-a \asymp 1$,
and $d_I-b \asymp k^{-\frac{1}{3} + \epsilon}$. Accordingly, 
for $0 \le p,q \le n$ and $1 \le j \le 2n$, we get
\begin{align*}
	\left| \mathcal{F}(a,b;a,b;c_{I},d_{I};n,p,q;j) \right|\
	\lesssim 1+ (k^{-\frac{1}{3} + \epsilon})^p \, \log k +(k^{-\frac{1}{3} + \epsilon})^q \, \log k 
	\lesssim 1 + \log k
\end{align*}
for $2n-(p+q+j) = -1$, and
\begin{align*}
	\left| \mathcal{F}(a,b;a,b;c_{I},d_{I};n,p,q;j) \right|
	\lesssim 1+ (k^{-\frac{1}{3} + \epsilon})^{2n-(q+j)+1} + (k^{-\frac{1}{3} + \epsilon})^{2n-(p+j)+1} 
	\lesssim 1+ (k^{\frac{1}{3}-\epsilon})^{n-1}
\end{align*}
for $2n-(p+q+j)\ne -1$. Using these estimates in \eqref{eq:Wexact-i-ii}, and upon noting that $d_{I} - c_{I} \asymp 1$ as $k \to \infty$,
we obtain
\[
	\mathcal{W}(k;n;a,b)^2
	\lesssim 1 + \delta_{n,1} \, \log k + H[n-2] \, (k^{\frac{1}{3}-\epsilon})^{n-1},
\]
and therefore \eqref{eq:ISest} follows.

For (ii), as $k \to \infty$, we have
$a - c_I \asymp k^{-\frac{1}{3} + \delta}$,
$b - c_I \asymp k^{-\frac{1}{3} + \epsilon}$,
$d_I - a \asymp 1$,
and $d_I - b \asymp 1$
so that, for $0 \le p,q \le n$ and $1 \le j \le 2n$, we get
\begin{align*}
	\left| \mathcal{F}(a,b;a,b;c_{I},d_{I};n,p,q;j) \right|\
	\lesssim 1+ (k^{-\frac{1}{3} + \delta})^p (k^{-\frac{1}{3} + \epsilon})^q \log k
	\lesssim 1
\end{align*}
when $2n-(p+q+j) = -1$, and
\begin{align*}
	\left| \mathcal{F}(a,b;a,b;c_{I},d_{I};n,p,q;j) \right|
	& \lesssim 1+ (k^{-\frac{1}{3} + \delta})^{p} \, (k^{-\frac{1}{3} + \epsilon})^{q}
	[(k^{-\frac{1}{3} + \delta})^{2n-(p+q+j)+1} + (k^{-\frac{1}{3} + \epsilon})^{2n-(p+q+j)+1}]
	\\
	& \lesssim 1+ (k^{-\frac{1}{3} + \delta})^{2n-j+1} \, (k^{\epsilon-\delta})^{q} + (k^{-\frac{1}{3} + \epsilon})^{2n-j+1} \, (k^{\delta-\epsilon})^{q}
	\\
	& \lesssim 1+ k^{-\frac{1}{3} + \delta} \, (k^{\epsilon-\delta})^{n} + k^{-\frac{1}{3} + \epsilon}
	\\
	& \lesssim 1+ k^{-\frac{1}{3} + \delta} \, (k^{\epsilon-\delta})^{n}
\end{align*}
when $2n-(p+q+j)\ne -1$. Using these two estimates in \eqref{eq:Wexact-i-ii}  and recalling $d_{I} - c_{I} \asymp 1$ as $k \to \infty$,
we therefore obtain
\[
	\mathcal{W}(k;n;a,b)^2
	\lesssim 1 + H[n-2] \, k^{-\frac{1}{3} + \delta} \, (k^{\epsilon-\delta})^{n}
\]
from which \eqref{eq:ISTest} follows.

As for (iii), Lemma~\ref{lemma:exactintegral} in Appendix~\ref{sec:auxiliary} entails
\begin{multline} \label{eq:W-iii}
	\mathcal{W}(k;n;a,b)^2
	= \sum_{\substack{0 \le p,q \le n \\ 1 \le j \le 2n}}
	\binom{4n-j-1}{2n-j} \binom{n}{p}\binom{n}{q} (-1)^{n}
	\\
	\times
	\left\{
		\dfrac{\mathcal{F}(a,t_{1};a,b;c_{S},d_{S};n,p,q;j)}{(d_{S}-c_{S})^{4n-j}}
		+ \dfrac{\mathcal{F}(t_{1},b;a,b;c_{I},d_{I};n,p,q;j)}{(d_{I}-c_{I})^{4n-j}}
	\right\},
\end{multline}
and we need to estimate $\mathcal{F}(a,t_{1};a,b;c_{S},d_{S};n,p,q,j)$
and $\mathcal{F}(t_{1},b;a,b;c_{I},d_{I};n,p,q;j)$. Considering the former, we have
\begin{align*}
	\mathcal{F}(a,t_{1};a,b;c_{S},d_{S};n,p,q,j)
	= ( c_S-a)^{p}(c_S-b)^{q} \log \big( \dfrac{t_1-c_S}{a-c_S} \big)
	+ ( a-d_S)^{p}(b-d_S)^{q} \log \big( \dfrac{d_S-a}{d_S-t_1} \big)
\end{align*}
for $2n-(p+q+j) = -1$, and 
\begin{align*}
	\mathcal{F}(a,t_{1};a,b;c_{S},d_{S};n,p,q,j)
	& = \dfrac{(c_S-a)^{p} (c_S-b)^{q}}{2n-(p+q+j)+1}
	\left[ (t_1-c_S)^{2n-(p+q+j)+1} - (a-c_S)^{2n-(p+q+j)+1} \right]
	\\
	& + \dfrac{(a-d_S)^{p}(b-d_S)^{q}}{2n-(p+q+j)+1}
	\left[ (d_S-a)^{2n-(p+q+j)+1} - (d_S-t_1)^{2n-(p+q+j)+1} \right]
\end{align*}
for $2n-(p+q+j) \ne -1$. Since
$c_{S}-a \asymp k^{-\frac{1}{3} + \delta}$, $|c_{S}-b| \lesssim k^{-\frac{1}{3} + \epsilon}$, $c_{S}-t_1 \asymp k^{-\frac{1}{3}}$, 
$d_{S}-a \asymp 1$, $d_{S}-b \asymp 1$, $d_{S}-t_1 \asymp 1$, and
$d_{S} - c_{S} \asymp 1$ so that, for $0 \le p,q \le n$ and $1 \le j \le 2n$, we get
\begin{align*}
	\mathcal{F}(a,t_{1};a,b;c_{S},d_{S};n,p,q,j)
	\lesssim 1 + (k^{-\frac{1}{3} + \delta})^p (k^{-\frac{1}{3} + \epsilon})^q \log k
	\lesssim 1
\end{align*}
for $2n-(p+q+j) = -1$, and
\begin{align*}
	\mathcal{F}(a,t_{1};a,b;c_{S},d_{S};n,p,q,j)
	& \lesssim 1 + (k^{-\frac{1}{3} + \delta})^p (k^{-\frac{1}{3} + \epsilon})^q
	\left[ (k^{-\frac{1}{3}})^{2n-(p+q+j)+1} + (k^{-\frac{1}{3} + \delta})^{2n-(p+q+j)+1} \right]
	\\ 
	& \lesssim 1+ (k^{-\frac{1}{3}})^{2n-j+1} (k^{\delta})^p (k^{\epsilon})^q
	+ (k^{-\frac{1}{3} + \delta})^{2n-j+1} (k^{\epsilon-\delta})^q
	\\ 
	& \lesssim 1+ k^{-\frac{1}{3}} \, (k^{\epsilon+ \delta})^n
	+ k^{-\frac{1}{3} + \delta} \, (k^{\epsilon-\delta})^n
	\\ 
	& \lesssim 1+ k^{-\frac{1}{3}} \, (k^{\epsilon+ \delta})^n
\end{align*}
for $2n-(p+q+j) \ne -1$. The same estimates hold also for $\mathcal{F}(t_{1},b;a,b;c_{I},d_{I};n,p,q;j)$.
Accordingly \eqref{eq:W-iii} implies
\[
	\mathcal{W}(k;n;a,b)^2 \lesssim 1+ k^{-\frac{1}{3}} \, (k^{\epsilon+ \delta})^n,
\]
and this yields \eqref{eq:SBest}. 
\end{proof}

In Theorem \ref{thm:ecevitozen}, for a given $n \in \mathbb{Z}_+$, taking $n_1 = \ldots = n_{4m} = n$ and setting $d_1 = \ldots = d_{4m} = d$ for any
positive integer $d \ge  n-1$, we see that in order to balance the errors in all the $4m$ regions uniformly for all $n$ (cf. \eqref{eq:algebraicestimate1}), we must have
\[
	\dfrac{1-3\epsilon_{1}}{6}
	= \epsilon_{m}
	= \dfrac{\epsilon_{j}-\epsilon_{j+1}}{2},
	\qquad
	 j = 1,\ldots,m-1.
\]
This system of equations can be explicitly solved to yield the following.

\begin{corollary} \label{corollary:algebraic1}
Given $n \in \mathbb{Z}_+$ and $m \in \mathbb{N}$, if $\epsilon_j$ are chosen as
\begin{equation}\label{eq:epsilons}
	\epsilon_{j} = \dfrac{1}{3} \, \dfrac{2m-2j+1}{2m+1},
	\qquad
	 j = 1,\ldots,m, 
\end{equation}
then, for all $k \ge k_0$ and $\mathbf{d} = (d,\ldots,d) \in \mathbb{N}^{4m}$ with $d \ge n-1$, we have
\begin{equation} \label{eq:est1}
	\Vert \eta - \hat{\eta}_\beta \Vert_{L^{2}(\partial K)}
	\lesssim \dfrac{C_k}{c_k}\,k^{-\frac{\beta}{3}} \, m\, \dfrac{1+ \delta_{n,1}\sqrt{\log k} + H[n-2] \, ( k^{\frac{1}{6m+3}})^{n-1}}{d^{n}}
\end{equation}
for the $\beta$-asymptotic frequency-adapted Galerkin approximation $\hat{\eta}_\beta \in \sigma_\beta + \mathcal{G}_{\mathbf{d}}$ to $\eta$.
\end{corollary}

To our knowledge, explicit analytical representations of the terms $a_{p,q,r,\ell}$ in the ansatz \eqref{eq:MT85Neumann} are not available and this corresponds to $\beta=0$.
However, when the number of subregions $m$ is chosen to increase proportional to
$\log k^{\frac{1}{6}}$, we observe that $k^{\frac{1}{6m+3}}$ is bounded and therefore \eqref{eq:est1} implies
\begin{equation} \label{eq:explainme} 
	\Vert \eta - \hat{\eta}_\beta \Vert_{L^{2}(\partial K)}
	\lesssim \, \dfrac{C_k}{c_k} \, \log k \dfrac{1+ \delta_{n,1}\sqrt{\log k}}{d^{n}}.
\end{equation}
Moreover, as $k \to \infty$, if the stability constant $\frac{C_k}{c_k}$ grows proportional to $k^{\varrho_1}$  for some $\varrho_1 > 0$, and 
$d$ is chosen to grow as $k^{\varrho_2}$ for some $\varrho_2 >0$, then
\[
	k^{\varrho_1} \, \frac{(\log k)^{\frac{3}{2}}}{k^{n\varrho_2}} \lesssim 1
\]
for all sufficiently large $n$. Since \eqref{eq:explainme} is valid for all $n$, we therefore deduce that the convergence of the method is spectral and requires
an increase of only $\mathcal{O}(k^{\epsilon})$ (for any $\epsilon >0$) in the total number of degrees of freedom to maintain accuracy for higher values of $k$.

One of the most important aspects of this method consists of incorporating sufficiently many terms $a_{p,q,r,\ell}$ in the ansatz \eqref{eq:MT85Neumann}
into the integral equation
in order to obtain a frequency independent method. Indeed, when $\beta \in \mathbb{Z}_+$ is chosen so that
$\frac{C_k}{c_k} \, k^{-\frac{\beta}{3}} (\log k)^{\frac{3}{2}} \lesssim 1$ as $k \to \infty$, then
\[
	\Vert \eta - \hat{\eta}_\beta \Vert_{L^{2}(\partial K)}
	\lesssim \dfrac{1}{d^{n}}.
\]
This shows that the method is not only spectral but also independent of frequency in the sense that prescribed accuracies can be attained with
the utilization of fixed numbers of degrees of freedom.
   
\subsection{$\beta$-asymptotic Galerkin boundary element method based on frequency dependent changes of variables}
\label{sec:method2}

For the construction of $\beta$-asymptotic Galerkin approximation spaces based on frequency dependent changes of variables,
given positive constants $\xi_j, \xi_j', \zeta_j,\zeta_j'$, $j=1,2$, satisfying
\begin{align*}
	t_1 + \xi_1
	\le t_1 + \xi_1'
	& = t_2 - \xi_2'
	\le t_2 - \xi_2,
	\\
	t_2 + \zeta_2
	\le t_2 + \zeta_2'
	& = 2P + t_1 - \zeta_1'
	\le 2P + t_1 - \zeta_1,
\end{align*}
we define, for any wavenumber $k>1$, the \emph{illuminated transition regions} as
\begin{align*}
	\mathcal{I}_1 = [a_1,b_1] = [t_1 + \xi_1 k^{-\frac{1}{3}}, t_1 + \xi_1'],
	\qquad
	\mathcal{I}_2 = [a_2,b_2] = [t_2 - \xi_2' , t_2 - \xi_2 k^{-\frac{1}{3}}],
\end{align*}
\emph{shadow transition regions} as
\begin{align*}
	\mathcal{I}_3 = [a_3,b_3] = [t_1 - \zeta_1', t_1 - \zeta_1 k^{-\frac{1}{3}} ],
	\qquad
	\mathcal{I}_4 = [a_4,b_4] = [t_2 + \zeta_2 k^{-\frac{1}{3}}, t_2 + \zeta_2'],
\end{align*}
and the \emph{shadow boundary regions} as
\begin{align*}
	\mathcal{I}_5 = [a_5,b_5] = [t_1 - \zeta_1 k^{-\frac{1}{3}}, t_1 + \xi_1 k^{-\frac{1}{3}}],
	\qquad
	\mathcal{I}_6 = [a_6,b_6] = [t_2 - \xi_2 k^{-\frac{1}{3}}, t_2 + \zeta_2 k^{-\frac{1}{3}} ].
\end{align*}
In what follows we identifty $L^{2}\left( \partial K \right)$ and $L^2 ( \cup_{j=1}^{6} \mathcal{I}_j )$ through the
parametrization $\gamma$. 

In order to capture the boundary layers of $\rho_\beta$ in the transition regions as implied by Theorem~\ref{thm:etamder},
we introduce the frequency dependent changes of variables $\phi_j : \mathcal{I}_j \to \mathcal{I}_j$ by setting
\begin{align*}
	\phi_1(s) & = t_1 + \varphi_1 \left( s \right) k^{\psi_1 \left( s \right)},
	\qquad
	\phi_2(s) = t_2 - \varphi_2 \left( s \right) k^{\psi_2 \left( s \right)},
	\\
	\phi_3(s) & = t_1 - \varphi_3 \left( s \right) k^{\psi_3 \left( s \right)},
	\qquad	
	\phi_4(s) = t_2 + \varphi_4 \left( s \right) k^{\psi_4 \left( s \right)},
	\\
	\phi_5(s) & = s,
	\hspace{3.33cm}
	\phi_6(s) = s.
\end{align*}
Here $\psi_j$ are constructed so as to linearly 
increase from $-\frac{1}{3}$ to $0$ as one moves away from the shadow boundaries, and
$\varphi_j$ are linear functions chosen to ensure that the maps $\phi_j : \mathcal{I}_j \to \mathcal{I}_j$
are bijective. They are defined explicitly as
\begin{align*}
	\psi_1(s) = -\dfrac{1}{3} \dfrac{b_1-s}{b_1-a_1},	
	\qquad
	\varphi_1(s) & = \xi_1 + \left( \xi_1'-\xi_1 \right) \dfrac{s-a_1}{b_1-a_1},
	\\
	\psi_2(s) = -\dfrac{1}{3} \dfrac{s-a_2}{b_2-a_2},
	\qquad
	\varphi_2(s) &= \xi_2' + \left( \xi_2-\xi_2' \right) \dfrac{s-a_2}{b_2-a_2},
	\\
	\psi_3(s) = -\dfrac{1}{3} \dfrac{s-a_3}{b_3-a_3},	
	\qquad
	\varphi_3(s) & = \zeta_1' + \left( \zeta_1-\zeta_1' \right) \dfrac{s-a_3}{b_3-a_3},
	\\
	\psi_4(s) = -\dfrac{1}{3} \dfrac{b_4-s}{b_4-a_4},
	\qquad
	\varphi_4(s) &= \zeta_2 + \left( \zeta_2'-\zeta_2 \right) \dfrac{s-a_4}{b_4-a_4}.
\end{align*}

With these definitions, we are now ready to introduce the Galerkin approximation spaces and the associated asymptotic solutions.

\begin{definition}
For $\mathbf{d} = \left( d_1, \ldots, d_6 \right) \in \mathbb{Z}_+^{6}$, the \emph{Galerkin approximation space based on 
frequency dependent changes of variables} of dimension $6 + |\mathbf{d}|$ in $L^2(\partial K)$ is defined as
\begin{equation} \label{eq:GalASCV}
	\mathcal{C}_{\mathbf{d}}
	= e^{i k \, \alpha \cdot \gamma} \, \mathcal{C}_{\mathbf{d}}^{\rm slow}
	= e^{i k \, \alpha \cdot \gamma} \bigoplus_{j=1}^{6}
	\charfunc_{\mathcal{I}_j} \ \Pol_{d_j} \circ \phi_j^{-1}.
\end{equation}
\end{definition}

\begin{definition}
Given $\beta \in \mathbb{Z}_+$, the \emph{$\beta$-asymptotic Galerkin approximation $\hat{\eta}_\beta$ to $\eta$ based on
frequency dependent changes of variables} is defined as
\begin{equation} \label{eq:defineetabeta}
	\hat{\eta}_\beta = \sigma_\beta + \hat{\rho}_\beta \in \sigma_\beta + \mathcal{C}_{\mathbf{d}}
\end{equation}
where $\hat{\rho}_\beta = e^{i k \, \alpha \cdot \gamma} \hat{\rho}_\beta^{\rm slow} \in \mathcal{C}_\mathbf{d}$
is the unique solution to the Galerkin formulation
\[
	\langle \hat{\mu}, \mathcal{R}_k \hat{\rho}_\beta \rangle = \langle \hat{\mu}, f_\beta \rangle,
	\qquad
	\text{for all }
	\hat{\mu} \in \mathcal{C}_{\mathbf{d}},
\]
of the integral equation \eqref{eq:inteqnew}.
\end{definition}

The convergence properties of the $\beta$-asymptotic Galerkin approximations $\hat{\eta}_\beta$ to $\eta$ based on frequency dependent changes of variables are as summarized in the next theorem.

\begin{theorem} \label{thm:eceviteruslu}
Given $(n_1,\ldots,n_6) \in \mathbb{Z}_+^6$, for all $k \ge k_0$ and $\mathbf{d} = (d_1,\ldots,d_6) \in \mathbb{N}^{6}$ with
$d_j \ge n_j-1$, we have
\[
	\Vert \eta - \hat{\eta}_\beta \Vert
	\lesssim \dfrac{C_k}{c_k} k^{-\frac{\beta}{3}} \sqrt{\log k} \,
	\Big( \sum_{j=1}^4 (\log k)^{n_j }(d_j)^{-n_j} +
	\sum_{j=5}^{6} (d_j)^{-n_j}
	\Big)
\]
for the $\beta$-asymptotic Galerkin approximation $\hat{\eta}_\beta \in \sigma_\beta + \mathcal{C}_{\mathbf{d}}$ to $\eta$
based on frequency dependent changes of variables. 
\end{theorem}

\begin{proof}
Arguing as in the proof of Theorem~\ref{thm:ecevitozen}, and then changing variables and observing that $0 < \phi_j' \lesssim \log k$ on
$\mathcal{I}_j$, we obtain
\begin{align*}
	\Vert \eta - \hat{\eta}_\beta \Vert
	& \le \dfrac{C_k}{c_k} \sum_{j=1}^{6} \inf_{p_j \in \mathbb{P}_{d_j}} \Vert \rho_\beta^{\rm slow} - p_j \circ \phi_j^{-1} \Vert_{L^2(\mathcal{I}_j)}
	\\
	& = \dfrac{C_k}{c_k} \sum_{j=1}^{6} \inf_{p_j \in \mathbb{P}_{d_j}}
	\Vert (\rho_\beta^{\rm slow} \circ \phi_j - p_j) \sqrt{\phi'_j} \Vert_{L^2(\mathcal{I}_j)}
	\\
	& \lesssim \dfrac{C_k}{c_k} \sqrt{\log k} \, \sum_{j=1}^{6} \inf_{p_j \in \mathbb{P}_{d_j}}
	\Vert \rho_\beta^{\rm slow} \circ \phi_j - p_j \Vert_{L^2(\mathcal{I}_j)}.
\end{align*}
Theorem~\ref{thm:pae} in Appendix~\ref{sec:auxiliary} therefore gives
\begin{align*}
	\Vert \eta - \hat{\eta}_\beta \Vert
	& \lesssim \dfrac{C_k}{c_k} \sqrt{\log k} \, \sum_{j=1}^{6}
	\left[ \int_{a_j}^{b_j} |D_s^{n_j}(\rho^{\rm slow}_\beta \circ \phi_j)(s)|^{2}(s-a_j)^{n_j}(b_j-s)^{n_j} ds \right]^{\frac{1}{2}} (d_j)^{-n_j}.
\end{align*}
Accordingly, since $b_j - a_j \asymp 1$ for $j=1,2,3,4$, and $\phi_5$ and $\phi_6$ are identity maps, the next lemma yields
\begin{align*}
	\Vert \eta - \hat{\eta}_\beta \Vert
	& \lesssim \dfrac{C_k}{c_k} k^{-\frac{\beta}{3}} \sqrt{\log k} \,
	\Big( \sum_{j=1}^4 (\log k)^{n_j }(d_j)^{-n_j} +
	\sum_{j=5}^{6}
	\mathcal{W}(k,n_j,a_j,b_j) \, (d_j)^{-n_j}
	\Big)
\end{align*}
where $\mathcal{W}$ is as defined in \eqref{eq:mathcalW}. Thus the result follows from part (iii) of Lemma~\ref{lemma:babap}.
\end{proof}

\begin{lemma} \label{lemma:dercomp}
Given $k \ge k_0$ and $n \in \mathbb{N}$, the estimates
\begin{equation} \label{eq:Dsnrhobetacompphijestlem}
	|D_s^n (\rho_\beta^{\rm slow} \circ \phi_j)|
	\lesssim k^{-\frac{\beta}{3}}
	\left\{
		\begin{array}{ll}
			(\log k)^n,
			& j=1,2,3,4,
			\\
			W(\cdot,k)^{-n},
			& j=5,6,
		\end{array}
	\right.
\end{equation}
hold on $\mathcal{I}_j$.
\end{lemma}
\begin{proof}
For $n=0$, the result is immediate from Theorem~\ref{thm:etamder}, so we assume $n \ge 1$.
For $j=1,\ldots, 4$, we utilize Fa\'{a} Di Bruno's formula for the derivatives of a composition \cite{Johnson02} to estimate
\[
	|D^n_s (\rho_\beta^{\rm slow} \circ \phi_j)|
	\lesssim \sum_{(m_{1},\ldots,m_{n}) \in \mathcal{F}_n}
	|(D^{m}_s \rho_\beta^{\rm slow}) (\phi_j)|
	\prod_{\ell=1}^{n} |D^{\ell}_s \phi_j|^{m_{\ell}}
\]
where $\mathcal{F}_n = \{ (m_1,\ldots,m_n) \in \mathbb{Z}^n_+ : n = \sum_{\ell=1}^{n} \ell m_{\ell} \}$
and $m = \sum_{\ell=1}^{n} m_{\ell}$. Since \cite[Proposition 4.3]{EcevitEruslu19}
\[
	|D^{\ell}_s \phi_j| \lesssim (\log k)^{\ell} k^{\psi_j},
\]
we therefore obtain
\begin{align}
	|D^n_s (\rho_\beta^{\rm slow} \circ \phi_j)|
	& \lesssim (\log k)^n
	\sum_{(m_{1},\ldots,m_{n}) \in \mathcal{F}_n}
	|(D^{m}_s \rho_\beta^{\rm slow}) (\phi_j)|
	k^{m\psi_j}
	\label{eq:Dsnrhobetacompphij}
	\\
	& \lesssim \left( \log k \right)^n
	\sum_{m=0}^{n} |(D^{m}_s \rho_\beta^{\rm slow}) (\phi_j)|
	k^{m\psi_j}
	\lesssim k^{-\frac{\beta}{3}} \, (\log k)^n
	\sum_{m=0}^{n} W(\phi_j,k)^{-m} k^{m\psi_j},
	\nonumber
\end{align}
where the last inequality is a consequence of Theorem~\ref{thm:etamder}.
Note that, since
\begin{align*}
	W(\phi_j,k)
	= k^{-\frac{1}{3}} + |\omega(\phi_j)|
	> |\omega(\phi_j)|
	= |(\phi_j-t_1)(t_2-\phi_j)|
\end{align*}
and
\[
	\begin{array}{lll}
		\phi_1-t_1 = \varphi_1k^{\psi_1} \ge \varphi_1(a_1)k^{\psi_1} = \xi_1k^{\psi_1}, 
		& t_2 - \phi_1 \ge t_2-\phi_1(b_1) = \xi'_2,
		& \text{on } \mathcal{I}_1,
		\\
		\phi_2 - t_1 \ge \phi_2(a_2) -t_1 = \xi'_1,
		& t_2 - \phi_2 = \varphi_2k^{\psi_2} \ge \varphi_2(a_2)k^{\psi_2} = \xi'_2k^{\psi_2}, 
		& \text{on } \mathcal{I}_2,
		\\
		t_1 - \phi_3 = \varphi_3k^{\psi_3} \ge \varphi_3(b_3)k^{\psi_3} = \zeta_1k^{\psi_3}, 
		& t_2 - \phi_3  > t_2 - t_1,
		& \text{on } \mathcal{I}_3,
		\\
		\phi_4 - t_1  > t_2 - t_1,
		& \phi_4 - t_2 = \varphi_4k^{\psi_4} \ge \varphi_4(a_4)k^{\psi_4} = \zeta'_2k^{\psi_4},
		& \text{on } \mathcal{I}_4,
	\end{array} 
\]
setting $\xi = \min \{ \xi_1\xi'_1, \xi'_1\xi'_2, \zeta_1(t_2-t_1),\zeta_2'(t_2-t_1) \}$, we have
\begin{equation} \label{eq:Winvkpsiest}
	W(\phi_j,k)^{-1} k^{\psi_j} \le \frac{1}{\xi^2}
	\quad
	\text{on }
	\mathcal{I}_j
\end{equation}
for $j=1,2,3,4$. Use of \eqref{eq:Winvkpsiest} in \eqref{eq:Dsnrhobetacompphij} therefore proves 
\eqref{eq:Dsnrhobetacompphijestlem} for $j=1,2,3,4$. Note that \eqref{eq:Dsnrhobetacompphijestlem}
is immediate from Theorem~\ref{thm:etamder} for $j =5,6$ since, in this case, $\phi_j$ is the identity map on $\mathcal{I}_j$.
This finishes the proof.
\end{proof}

Assigning the same local polynomial degree $d$ to each interval $\mathcal{I}_j$, we obtain the following.

\begin{corollary} \label{cor:chvar}
Given $n \in \mathbb{Z}_+$, for all $k \ge k_0$ and $\mathbf{d} = (d,\ldots,d) \in \mathbb{N}^{6}$ with $ d \ge n-1$, we have 
\begin{equation} \label{eq:newalgerror}
	\Vert \eta - \hat{\eta}_\beta \Vert_{_{L^2(\partial K)}}
	\lesssim
	\dfrac{C_k}{c_k} \, k^{-\frac{\beta}{3}} \, \dfrac{\left( \log k \right)^{n+\frac{1}{2}}}{d^n} 
\end{equation}
for the $\beta$-asymptotic Galerkin approximation $\hat{\eta}_\beta \in \sigma_\beta + \mathcal{C}_{\mathbf{d}}$ to $\eta$ based on frequency dependent changes of variables.
\end{corollary}

In the case that the stability constant  $C_k/c_k$ grows algebraically (i.e. $C_k/c_k \asymp k^{\delta}$ for some $\delta >0$) as $k \to \infty$,
Corollary~\ref{cor:chvar} implies the followings. First, recall that $\beta = 0$ means no term $a_{p,q,r,\ell}$ in the ansatz \eqref{eq:MT85Neumann} is incorporated into the integral equation \eqref{eq:inteq}. Still, even in this case, estimate \eqref{eq:newalgerror} clearly implies that the method is not
only spectral for each fixed $k$, but also increasing the number of degrees of freedom proportional to $k^{\epsilon}$ (for any $\epsilon >0$) is sufficient to fix the
approximation error with increasing $k$. In other words, the method is spectral and \emph{almost} frequency independent. When $\beta >0$ is chosen so that
$\delta - \beta/3 <0$, the method is still spectral for each fixed $k$, and moreover it is frequency independent in the sense that prescribed accuracies can
be attained with the utilization of frequency independent numbers of degrees of freedom with increasing $k$.

\section{Numerical results} \label{sec:5}

In this section, we present numerical results validating the theoretical developments on the \emph{$\beta$-asymptotic Galerkin boundary element method
based on frequency dependent changes of variables} (GBemCV) described in \S\ref{sec:method2}. Since the explicit forms of the terms $a_{p,q,r,\ell}$
in the asymptotic expansion \eqref{eq:MT85Neumann} are not available over the entire boundary $\partial K$, we take $\beta =0$.
The numerical results obtained using the \emph{frequency-adapted $\beta$-asymptotic Galerkin boundary element method} detailed in \S\ref{sec:method1}
are entirely similar, and therefore not presented.

Several integral equations can be used to solve the Neumann problem \eqref{eq:Helmholtz}, and they are expressed in the general form  
\begin{equation} \label{eq:geninteq}
	\mathcal{R}_k \eta = f_k
	\quad
	\text{on }
	\partial K
\end{equation}
where $\eta$ is the total field \cite{ColtonKress83}. Using standard techniques, the operator $\mathcal{R}_k$ and the right hand side  $f_k$ can be taken as 
\begin{align}
	\mathcal{R}_k & = \mathcal{I}-2\mathcal{K}_k,
	\hspace{2.54cm}
	f_k = 2u^{\rm inc} \label{eq:1}
	\\
	\mathcal{R}_k & = \mathcal{T}_k,
	\hspace{3.52cm}
	f_k = \partial_{\nu}u^{\rm inc}
	\nonumber
	\\
	\mathcal{R}_k & = \mathcal{I} -2(\mathcal{K}_k + i \mu \mathcal{T}_k),
	\hspace{1cm}
	f_k = 2(u^{\rm inc} - i \mu \partial_{\nu} u^{\rm inc})
	\nonumber
\end{align}
where $\mu$ is a real coupling parameter, and the double layer and hyper-singular operators
are given respectively by
\[
	\mathcal{K}_k \phi(x) = \int_{\partial K} \dfrac{ \partial G_k(x,y)}{\partial \nu(y)} \, \phi(y) \, ds(y),
	\qquad
	\mathcal{T}_k \phi(x) = \dfrac{\partial}{\partial \nu(x)} \int_{\partial K} \dfrac{ \partial G_k(x,y)}{\partial \nu(y)} \, \phi(y) \, ds(y).
\]
For the numerical tests performed here, we choose to implement the integral equation \eqref{eq:geninteq} using \eqref{eq:1} owing to the fact that
it contains only integral operators with weakly singular kernels. We discretize the operator $\mathcal{R}_k$ through use of the trapezoidal rule and
the Nystr\"{o}m method \cite{ColtonKress92} utilizing $10$ to $12$ points per wavelength. Regarding the implementation, we refer to
\cite{EcevitOzen17,EcevitEruslu19} since it is similar to its Dirichlet version.

We consider two different single-scattering configurations consisting of the unit circle $\{(\cos t,\sin t):  \, t \in [0,2\pi] \}$,
and the ellipse $\{(\frac{3}{2}\cos t,\frac{1}{2} \sin t):  \, t \in [0,2\pi] \}$ rotated by $\frac{\pi}{6}$ radians in the
counterclockwise direction. In both cases, the illumination is coming in from the left as depicted
in Figure~\ref{fig:configs}. 
\begin{figure}[htbp]
	\begin{center}		
		\includegraphics[height=4.9cm, trim={1.9cm 0.7cm 2cm 1cm},clip]{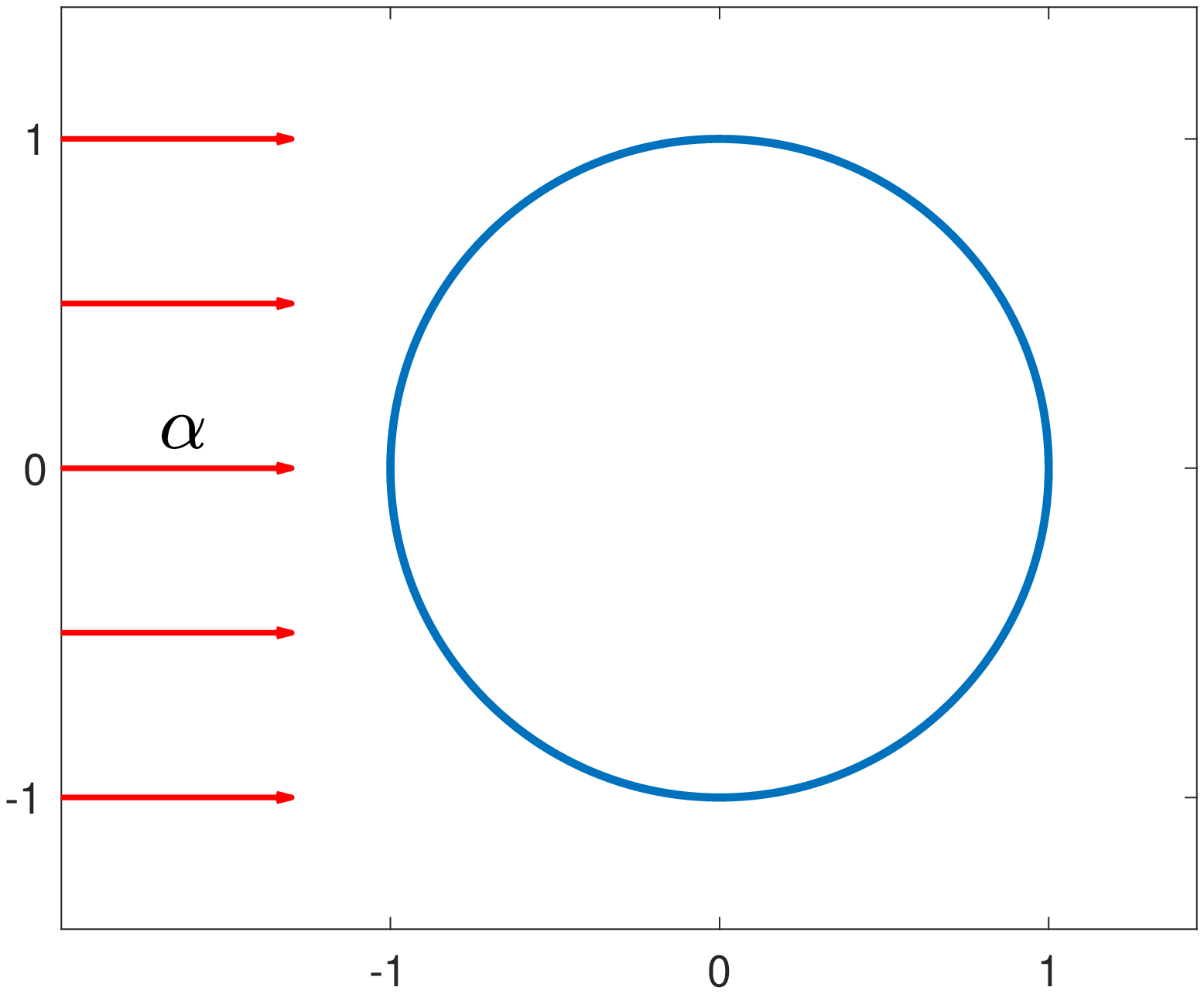}
		\qquad
		\includegraphics[height=4.9cm, trim={1.8cm 0.7cm 1.5cm 1cm},clip]{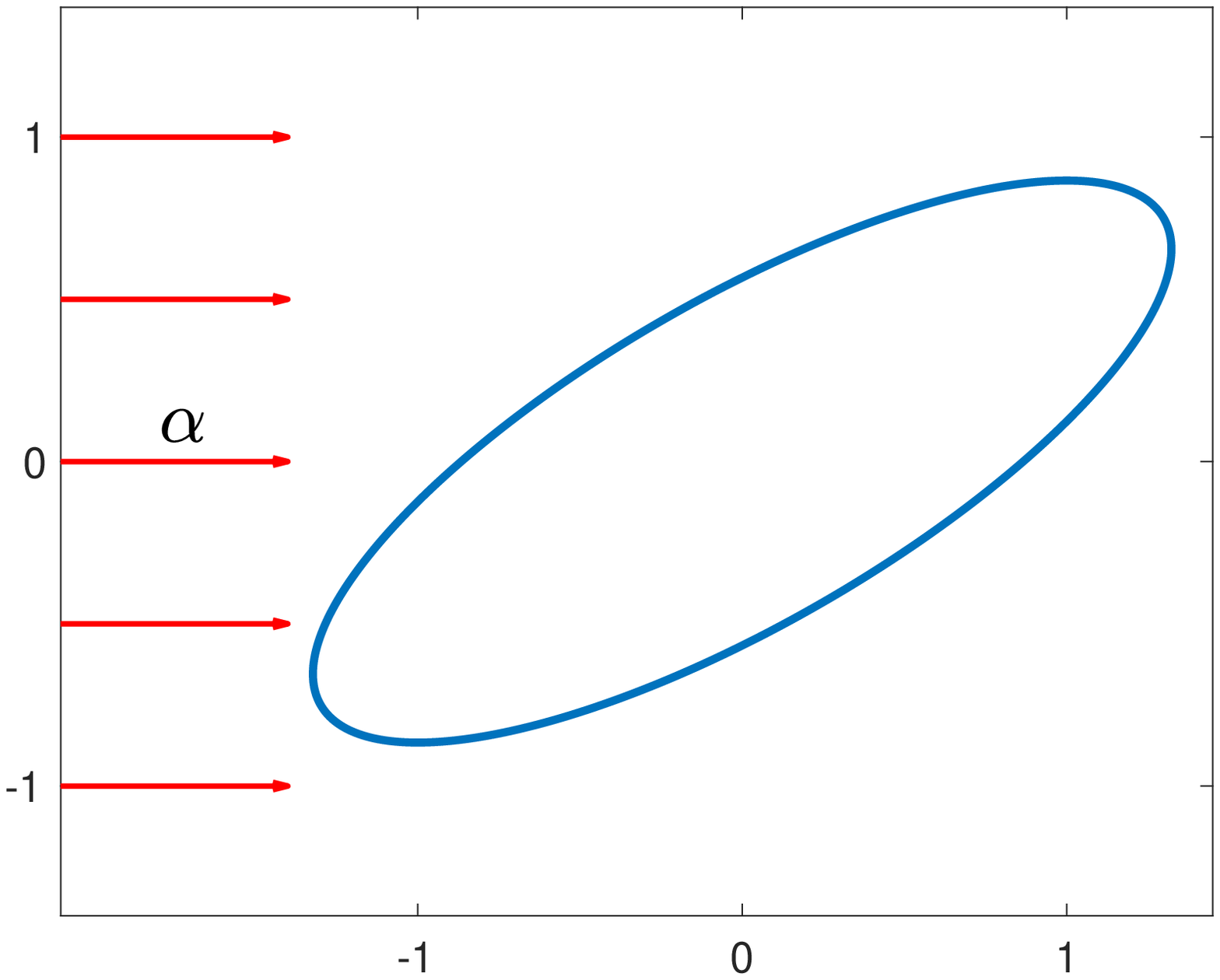} \vspace{-0.5cm}
	\end{center}
	\caption{Configurations used in the numerical tests.}
	\label{fig:configs}
\end{figure}

The unit circle is the standard test case since the analytical solution can be derived
with the aid of Fourier analysis and Jacobi-Anger expansion \cite{ColtonKress92}. For a circle of radius $r$,
switching to polar coordinates, the analytical solution for the Neumann boundary value problem is expressed as
\begin{equation} \label{eq:analytical-solution}
	\eta(\theta)
	= e^{ik r \cos \theta} + 
	\sum_{m \in \mathbb{Z}} i^{m+2} \frac{(J_{m} (kr))'}{(H^{(1)}_{m}(kr))'} \, e^{im\theta} \, H^{(1)}_m(kr).
\end{equation}
We display in Figure~\ref{fig:circle-solution} the real and imaginary parts of the analytical solution $\eta$ 
along with those of the envelope $\eta^{\rm slow}$ for $k = 50, 100, 200, 400, 800$. We observe that the oscillations in the solution $\eta$ are amplified with
increasing frequency. In addition, as the asymptotic theory predicts, $\eta^{\rm slow}$ is non-oscillatory
in the illuminated region, admits boundary layers around the shadow boundaries, and decays exponentially with increasing frequency in the deep shadow region.
In order to validate the accuracy of the approximations, we compare in Figure~\ref{fig:Yassine} the exact value and the numerical approximation of
$\eta$ for $k=400$. We can see that the GBemCV solution successfully approximates the exact solution over the entire boundary.

We analyze now the error produced by the use of the GBemCV method. The strategy resides in evaluating the error with respect to increasing
values of the local polynomial degree $d$ (see Corollary~\ref{cor:chvar}). In Figure~\ref{fig:error} (left), we plot the logarithmic $L^{2}$-errors
\begin{equation} \label{eq:logerror}
	\log_{10} ( \Vert \eta-\hat{\eta}_{\beta} \Vert_{L^2})
\end{equation}
for $d = 4,8,12,16,20$ and $k = 50, 100, 200, 400, 800$. Let us mention that $\hat{\eta}_{\beta}$ is defined in \eqref{eq:defineetabeta}, and $\eta$ represents the analytical
solution in \eqref{eq:analytical-solution}. In our numerical tests, we have constructed the Galerkin approximation spaces $\mathcal{C}_{\mathbf{d}}$ in \eqref{eq:GalASCV} utilizing
the same local polynomial degree $d$ on each of the six direct summands which generates the total number of degrees of freedom
$6(d+1)$. The results in Figure~\ref{fig:error} (left) show that, for any fixed value of $k$, the accuracy increases with increasing values of
$d$, and the method is frequency independent. 

For the elliptical configuration displayed in Figure~\ref{fig:configs}, since the analytical solution is not available, we first start by displaying in
Figure~\ref{fig:ellipse-solution} the numerical solutions obtained by the GBemCV method (for $\beta = 0$). We can observe that
$\eta$ and $\eta^{\rm slow}$ exhibit properties similar to the circle case. More precisely, with increasing wavenumber, $\eta^{\rm slow}$ is non-oscillatory
in the illuminated region, admits boundary layers around the shadow boundaries, and it decays exponentially in the deep shadow region.  
Concerning the error analysis, we plot in Figure~\ref{fig:error} (right) the local polynomial degree $d$ versus logarithmic $L^{2}$-errors
for $d = 4,8,12,16,20$ and $k = 50, 100, 200, 400, 800$. As mentioned above, in \eqref{eq:logerror}, $\hat{\eta}_{\beta}$ is defined in \eqref{eq:defineetabeta}, however,
$\eta$ is computed using the Nystr\"{o}m method \cite{ColtonKress92}. As in the case of the circle, the results in Figure~\ref{fig:error} (right) show that for any wavenumber
the accuracy increases with the increasing local polynomial degree $d$, and the method is also frequency independent. 
 
\begin{figure}[htbp]
	\centering
	\includegraphics*[width=5.9in,viewport=110 85 900 1185,clip]{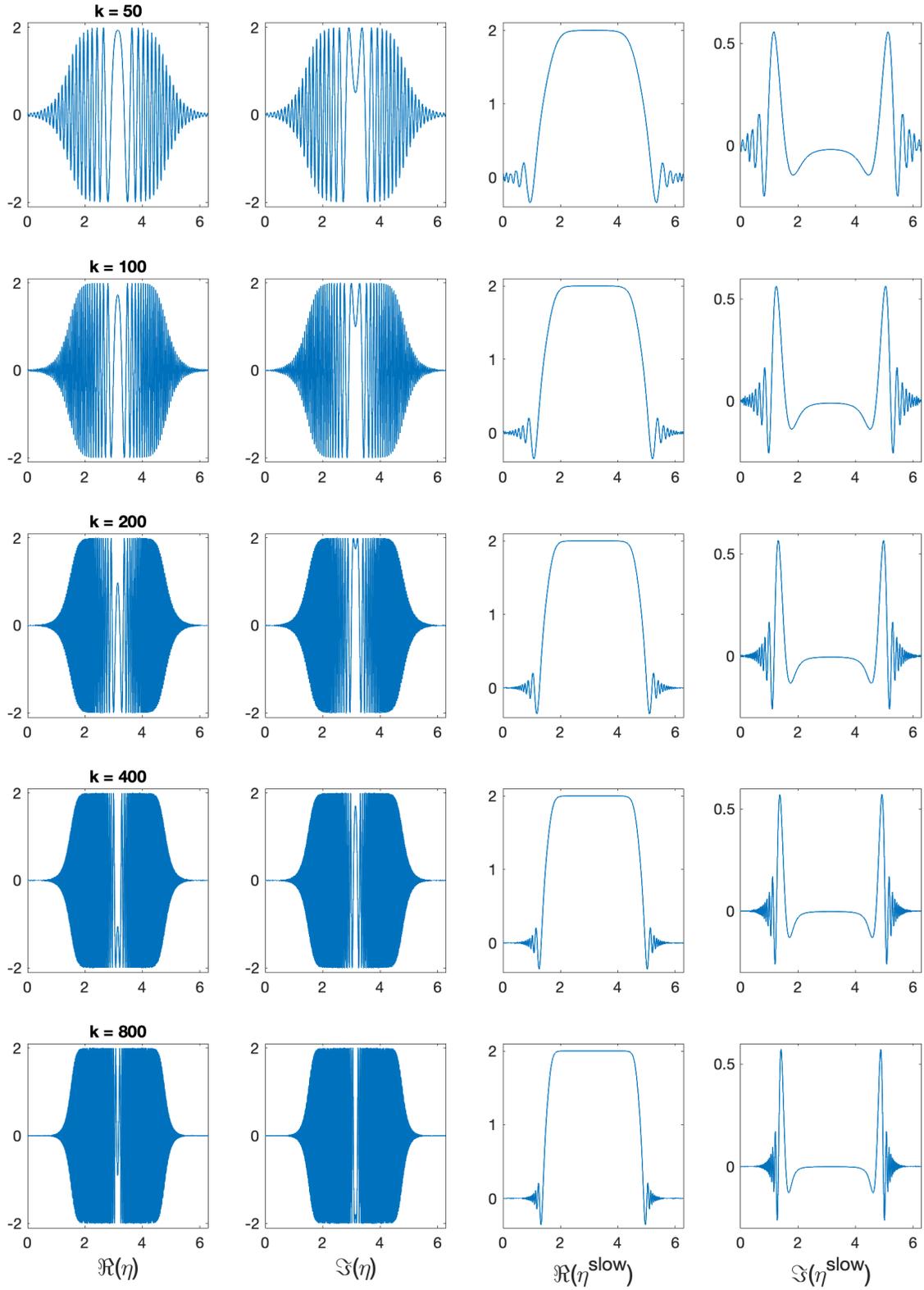}		
	\caption{Real (first column) and imaginary (second column) parts of $\eta$, and those of $\eta^{\rm slow}$ (columns three and four)
	associated with the circular scatterer in Figure~\ref{fig:configs} for $k = 50, 100, 200, 400, 800$.}
	\label{fig:circle-solution}
\end{figure}

\begin{figure}[htbp]
	\centering
	\includegraphics*[width=6.5in,viewport=60 30 840 440,clip]{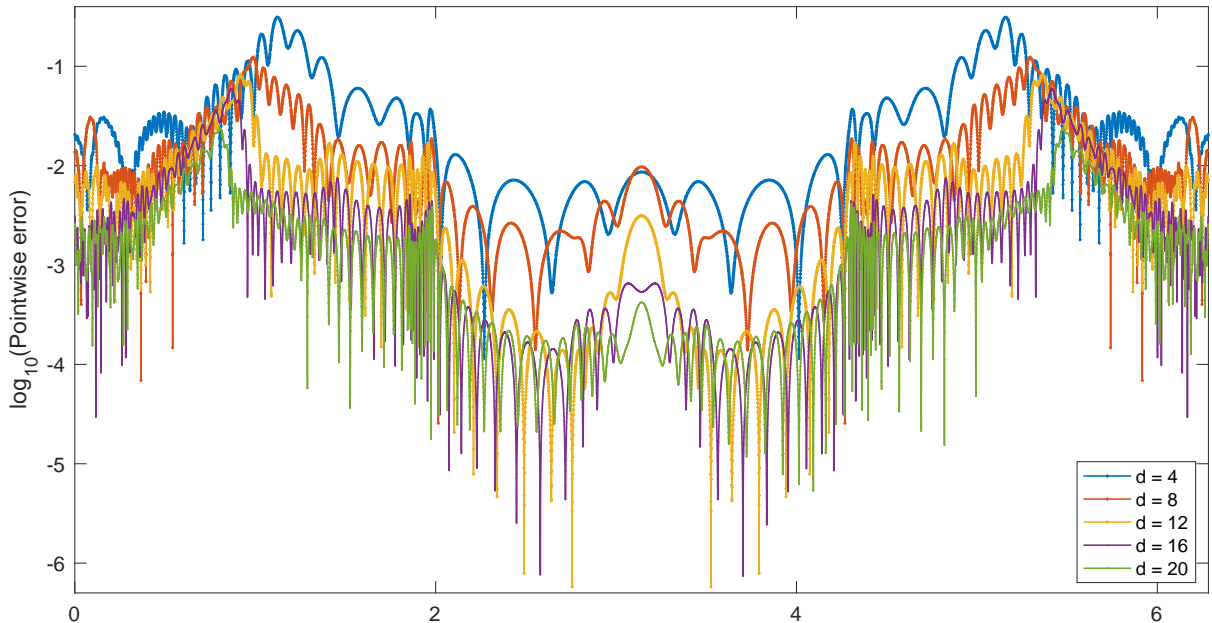}		
	\caption{Logarithmic pointwise errors for the unit circle in Figure~\ref{fig:configs} for $k = 400$ and local polynomial degrees $d = 4,8,12,16,20$.}
	\label{fig:Yassine}
\end{figure}

\begin{figure}[htbp]
	\centering
	\includegraphics*[width=5.9in,viewport=110 85 900 1185,clip]{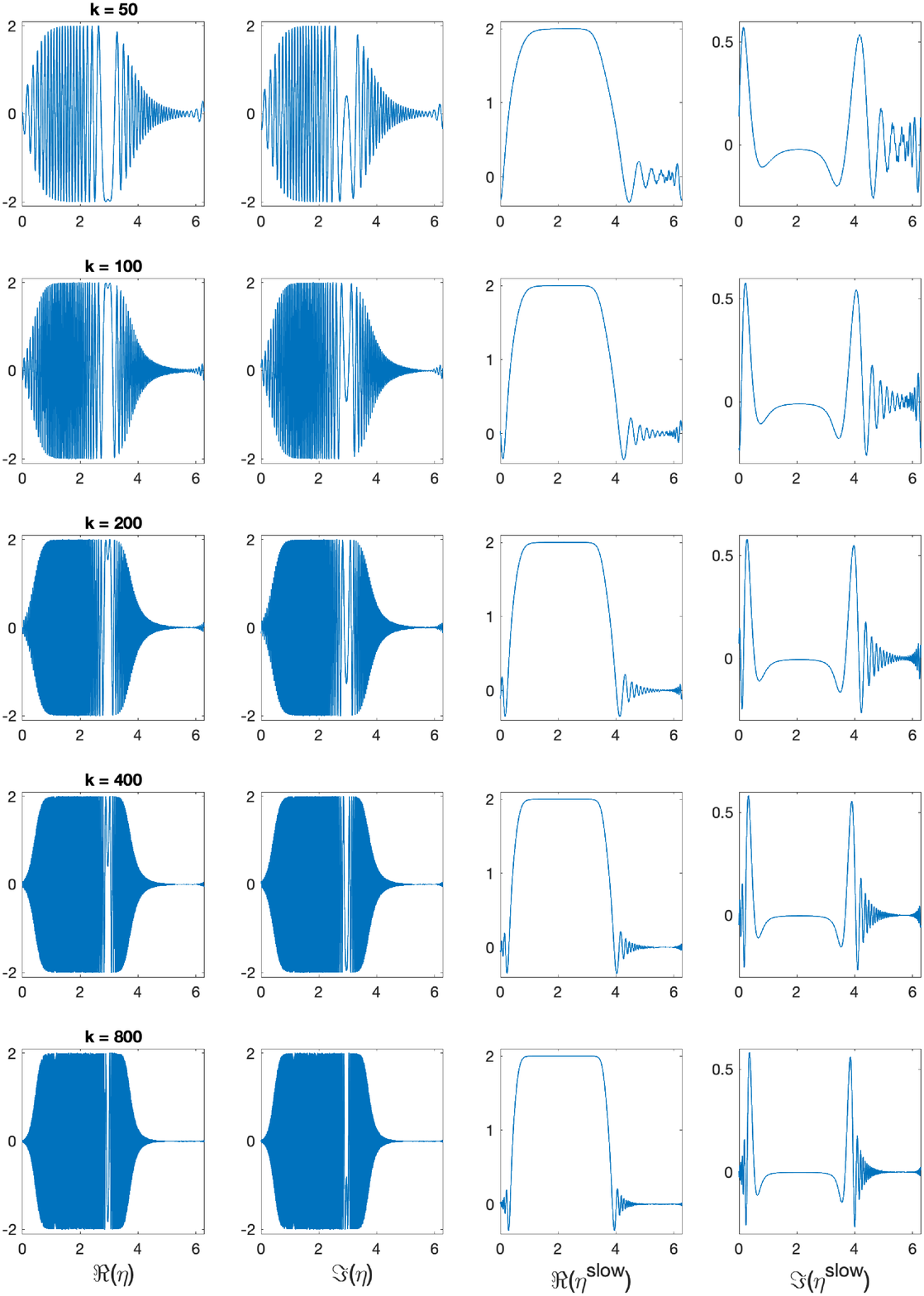}		
	\caption{Real (first column) and imaginary (second column) parts of $\eta$, and those of $\eta^{\rm slow}$ (columns three and four)
	associated with the elliptical scatterer in Figure~\ref{fig:configs} for $k = 50, 100, 200, 400, 800$.}
	\label{fig:ellipse-solution}
\end{figure}

\begin{figure}[htbp]
	\centering
	\subfigure[Unit circle]{
	\includegraphics*[width=3.0in,viewport=5 0 510 400,clip]{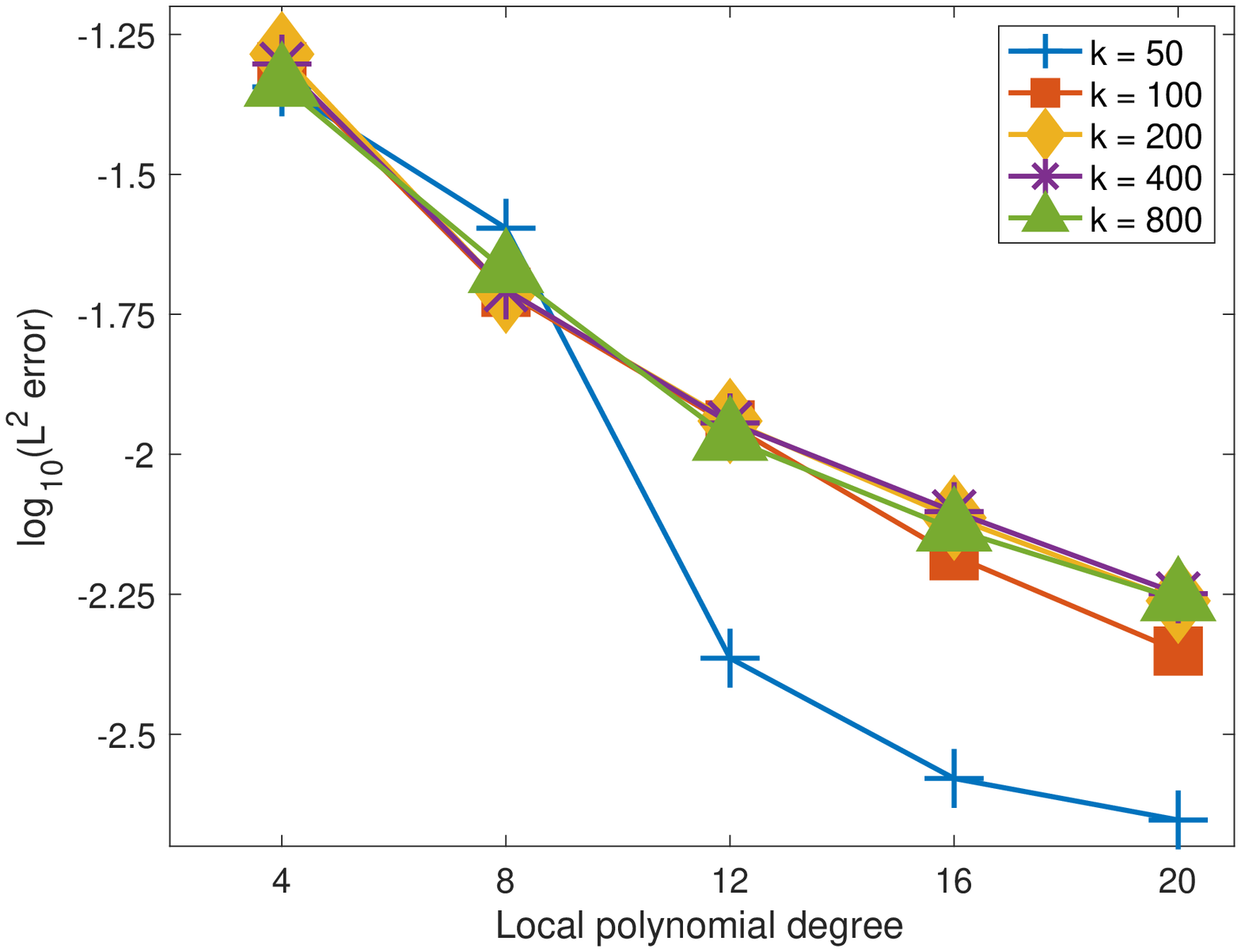}}
	\subfigure[Ellipse]{
	\includegraphics*[width=3.0in,viewport=5 0 510 400,clip]{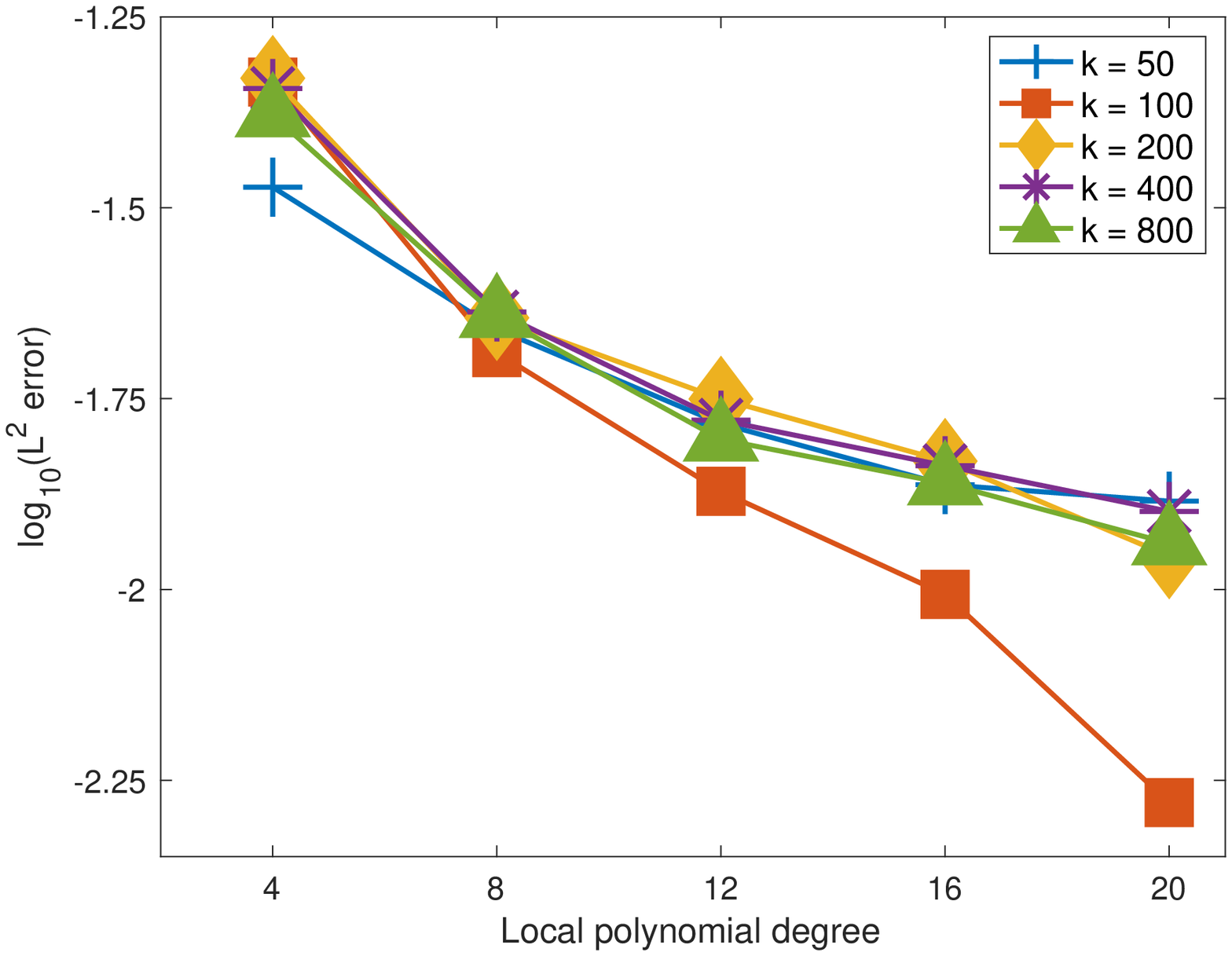}}		
	\caption{Local polynomial degree versus $\log_{10}$($L^2$ error) for the unit circle (left) and the ellipse (right) for 
	$k = 50, 100, 200, 400, 800$.}
	\label{fig:error}
\end{figure}

\section{Conclusion} \label{sec:6}

In this paper, we proposed two different $\beta$-asymptotic Galerkin boundary element methods for approximations of solutions to high-frequency
sound-hard scattering problems. These two methods are based on the ansatz describing the asymptotic behavior of the total field associated with
the Neumann boundary condition. We provided the missing parts in the derivation of this ansatz given in \cite{MelroseTaylor85}.
For the convergence analyses, we used the ansatz to derive wavenumber explicit estimates on the derivatives of the slow envelope corresponding
to the total field on the boundary. An important ingredient of these analyses resides in the use of appropriate number of terms in the asymptotic expansion.
This resulted in the design of Galerkin boundary element methods that deliver prescribed
error tolerances using frequency independent numbers of degrees of freedom.
The numerical implementations confirm that the solutions corresponding to any fixed number of degrees
of freedom yield frequency independent approximations.

\appendix


\section{Asymptotic expansion of the total field} \label{sec:asymptotics} 

The ansatz representing the asymptotic expansion of the total field $\eta$ for the Neumann boundary value problem used in the present paper is given by Melrose and Taylor
in \cite{MelroseTaylor85}. However, the authors did not present all the mathematical steps needed in its derivation. In this section, we provide the missing details of this analysis.

Let $K \subset \mathbb{R}^{n+1}$ be a compact strictly convex obstacle such that $B= \partial K \subset \mathbb{R}^{n+1}$ is a smooth
hyper-surface, and consider the Neumann-to-Dirichlet operator
\begin{equation}
	N^{-1} : \mathcal{E}'({\mathbb{R}\times B}) \ni f(t,x)\mapsto v(t,x)|_{\mathbb{R}\times B} \in \mathcal{D}'({\mathbb{R}\times B})   
\end{equation}
where $\mathcal{D}'(B \times \mathbb{R})$ and $\mathcal{E}'({\mathbb{R}\times B})$ are the spaces of distributions and compactly
supported distributions respectively, and $v$ is the solution to the wave problem 
\begin{equation}\label{prob}
	\left\{
		\begin{array}{ll}
			(\partial_{tt} -\Delta ) v(t,x)=0 & \text{in } \mathbb{R}\times \Omega,
			\\
			\partial_{\nu}v(t,x)|_{\mathbb{R}\times B} = f(t,x) & \text{on } \mathbb{R} \times B,
		\end{array}
		\right.
\end{equation}
wherein $\Omega = \mathbb{R}^{n+1} \setminus K$ is the exterior domain, and $\nu$ is the outward unit normal.

In what follows, for an incident field $v^{i}(t,x) = \delta(t - \alpha \cdot x)$ ($\delta$ is the Dirac function)
with direction $\alpha \in \mathcal{S}^n$, we denote the solution of the wave problem \eqref{prob} associated with
$f(t,x) = (\alpha \, \cdot \, \nu(x)) \partial_t v^{i}(t,x)$ by $v$. In this case, the total field $v^t = v + v^{i}$ can be expressed on the boundary as \cite[p.296]{MelroseTaylor85}
\begin{equation}\label{sol on boundary}
	v^t(t,x)|_{\mathbb{R} \times B}
	= (I+N^{-1}(\alpha \cdot \nu(x)) \partial_t) \delta(t - \alpha \cdot x)
\end{equation} 
where $I$ is the identity operator. Using the same notation and procedure in \cite{MelroseTaylor85}, we define the \emph{Kirchhoff operator}
(see \cite[Equation 8.26]{MelroseTaylor85})
\[
	Q_N : \mathcal{E'}(\mathcal{S}^n \times \mathbb{R}) \ni v(t,x)
	\mapsto
	(I+N^{-1}(\alpha \cdot \nu(x) ) \partial_t )Fv(t,x) \in \mathcal{D'}(B \times \mathbb{R})
\]	
where $F$ is the Fourier integral operator \cite[Equation 9]{LazerguiBoubendir17}
\begin{equation} \label{FF}
	Fv(t,x) = \int_{\mathbb{R}\times \mathcal{S}^n} \kappa_F(t-s,w,x) v(s,\alpha) ds d \alpha 
\end{equation}
with kernel $\kappa_{F}(t,\alpha,x)=\delta(t-\alpha\cdot x)$.

As shown in \cite{MelroseTaylor85}, the asymptotic behavior of the total field $\eta$ is determined by the kernel
$\kappa_{Q_N}$ of the Kirchhoff operator $Q_N$.

\begin{lemma} \label{lemma:MT} \emph{\cite[Lemma 9.1]{MelroseTaylor85}}
The asymptotic behavior as $k \rightarrow \pm \infty$ of the total field $\eta(\alpha,k,x)$ obtained by inverse Fourier transformation of the kernel $\kappa_{Q_N}(\alpha,t,x)$
\begin{equation}
	\eta(\alpha,x,k)=\int e^{it k} \kappa_{Q_N}(\alpha,t,x) dt
\end{equation}
is determined by the singularities of the kernel $\kappa_{Q_N}$ modulo rapidly decreasing terms.
\end{lemma}

The main goal is therefore the study of the kernel $\kappa_{Q_N}$. To this end, one first utilizes the general theory of
\emph{Fourier integral operators with folding canonical relations} \cite[\S5 and \S6]{MelroseTaylor85} 
to decompose the operator $Q_N$.

\begin{theorem} \emph{\cite[Theorem 8.30]{MelroseTaylor85}}
The Kirchhoff operator $Q_N$ can be expressed as
\begin{equation} \label{Q_N}
	Q_N = J_1D\mathcal{A}^{-1} J_2
\end{equation}
where $J_1$ and $J_2$ are elliptic Fourier integral operators of order zero, 
$D\in OPS^{-\frac{n}{2}-\frac{1}{6}}_{\frac{1}{3},0}$ has an asymptotic expansion 
\begin{equation} \label{eq:D}
	D \sim \sum_{r \in \mathbb{Z}_+, \, \ell \in -\mathbb{N}} A_{r,\ell} \Phi^{r,\ell}
\end{equation}
with $A_{r,\ell} \in S^{-\frac{n}{2}-\frac{\ell}{3}-\frac{r}{3}+(\ell+1)_-}_{cl}$ and
$\Phi^{r,\ell}(k^{-\frac{1}{3}}\xi_1)\in S^{(\frac{\ell}{3}-\frac{2r}{3})_{+}}_{\frac{1}{3},0} (\mathbb{R})$
$($see \emph{\cite[p.11-12]{Melrose78}}$)$ so that
\begin{equation}
	Dv(t,x)
	\sim \sum_{r \in \mathbb{Z}_+, \, \ell \in -\mathbb{N}}
	\int e^{i(x-y) \cdot \xi + i(t-t')k} a_{r,\ell}(t,x,k,\xi) \Phi^{r,\ell}(k^{-\frac{1}{3}}\xi_{1}) v(t',y) dt' dy dk d\xi
\end{equation}
where 
$(k,\xi)$ are variables dual to $(t,x)$, and
$a_{r,\ell} \in S^{-\frac{n}{2}-\frac{r}{3}-\frac{\ell}{3}+(\ell+1)_-}_{1,0}$
admits an asymptotic expansion 
\begin{equation}
	a_{r,\ell}(t,x,k,\xi)
	\sim \sum_{q \in \mathbb{Z}_+} k^{-\frac{n}{2}-q-\frac{r}{3}-\frac{\ell}{3}+(\ell+1)_{-}} \, a_{q,r,\ell}(t,x,\xi)
\end{equation}
wherein $a_{q,r,\ell}$ are $C^{\infty}$ functions uniformly bounded together with all their derivatives
$($cf. \emph{\cite[Definition 2.5.6]{Martinez02}}$)$,
and $\mathcal{A}^{-1}$ is the convolution operator defined by Fourier transformation as
\emph{\cite[Equation 1.36]{MelroseTaylor85}}
\[
	\widehat{\mathcal{A}^{-1} v(t,x)} (\xi) = \dfrac{\hat{v}(\xi)}{A_+(k^{-\frac{1}{3}}\xi_1)}
\]
where $A_+(z) = Ai(e^{\frac{2\pi i}{3}}z)$ and $Ai$ is the Airy function \emph{\cite{Melrose75}}.
\end{theorem}

The Fourier integral operator $J_2$ is given by
\begin{align*}
	J_2 : \
	& \mathcal{E}'(\mathbb{R} \times \mathcal{S}^n) \to \mathcal{D}'(\mathbb{R} \times \mathbb{R}^n)
	\\
	& v(s,\alpha) \mapsto
	(J_2v)(t',y) = \int e^{i(y-\alpha) \cdot \xi-i(t'-s)k}
	a_{J_2}(s,\alpha,t',y) v(s,\alpha)
	ds d\alpha dk d\xi
\end{align*}
where  $a_{J_2} \in S^0_{1,0}$ ($a_{J_2}$ does not depend on $\xi $ and $k$ because it is a symbol of order $0$).
Applying the Dirac function at the base point $(0,\bar{\alpha})$ (see \cite{MelroseTaylor85,LazerguiBoubendir17}) 
yields
\begin{align}
	(J_2\delta_{(0,\bar{\alpha})})(t',y)
	\nonumber
	& = \int e^{i(y-\alpha) \cdot \xi-i(t'-s)k}
	a_{J_2}(s,\alpha,t',y)
	\delta_{(0,\bar{\alpha})}(s,\alpha)
	ds d\alpha dk d\xi
	\\
	& = \int e^{i(y-\bar{\alpha}) \cdot \xi -i t'k}
	a_{J_2}(\bar{\alpha},t',y) dk d\xi
	\nonumber
	\\
	& = \int e^{i(y-\bar{\alpha}) \cdot \xi -i t'k}
	a_{J_2}(\bar{\alpha},t',y)\widehat{\delta_{0}}(k,\xi)dk d\xi
	= (P\delta_{0})(t',y)
	\label{eq:Kdelta}
\end{align}
where the operator $P\in OPS^{-\frac{n}{2}+\frac{1}{6}}$ is specified by
\[
	(Pv)(t',y) = \int e^{i(y-\bar{\alpha}) \cdot \xi -i t'k}a_{J_2}(\bar{\alpha},y,t')\widehat{v}(k,\xi)dk d\xi.
\]
\noindent
Accordingly, use of \eqref{eq:Kdelta} in \eqref{Q_N} implies
\begin{equation} \label{QN operator}
	Q_N(\delta_{(0,\bar{\alpha})}(t,x))
	= J_1D\mathcal{A}^{-1}J_2(\delta_{(0,\bar{\alpha})}(t,x))
	= J_1D\mathcal{A}^{-1}P(\delta_{0}(t,x)).
\end{equation}
Finally, since $P\delta_0 = P^{\#}\delta_0 \mod OPS^{-\infty}$
and $P^{\#}$ commutes with $\mathcal{A}^{-1}$ \cite[p.295]{MelroseTaylor85},
\eqref{QN operator} can be rewritten as
\begin{equation} \label{eq:QNgood}
	Q_N(\delta_{(0,\bar{\alpha})}(t,x))
	= J_1DP^{\#}\mathcal{A}^{-1}(\delta_{0}(t,x))
\end{equation}
modulo rapidly decreasing terms.

In what follows, we briefly explain how representation \eqref{eq:QNgood} can be used
to express the amplitude associated with the kernel $\kappa_N$ as an asymptotic series of
oscillatory integrals each of which is amenable to an application of the stationary phase method \cite{Fedoryuk71}.

To this end, we first use \eqref{eq:D} to deduce for the composition $DP^{\#}$ of the pseudo-differential operators $D$ and $P^{\#}$
\begin{equation} \label{composition}
	DP^{\#}v(t,x)
	\sim \sum_{r \in \mathbb{Z}_+, \, \ell \in -\mathbb{N}}
	\int e^{i(x-y) \cdot \xi-i(t-t')k}a_{r,\ell}(t,x,k,\xi)
	\Phi^{r,\ell}(k^{-\frac{1}{3}} \xi_1)P^{\#}v(t',y)dt'dydk d\xi
\end{equation}
with
\begin{align}
	P^{\#}v(t',y)
	& = \int e^{i(y-\bar{\alpha}) \cdot \eta+it'\tau}
	p^{\#}(\bar{\alpha},t',y)
	\widehat{v}(\tau,\eta)
	d\tau d\eta
	\nonumber
	\\
	& = \int e^{i(y-z-\bar{\alpha}) \cdot \eta+i(t'-t'')\tau}p^{\#}(\bar{\alpha},t',y) v(t'',z)
	dt''dz d\tau d\eta
	\label{P2 definition}
\end{align}
where $p^{\#} \in S^0_{1,0}$,
and $(\tau,\eta)$ is the dual variable to $(t',y)$. Using \eqref{P2 definition} in \eqref{composition}, we therefore get 
\begin{equation} \label{eq:DPsharp}
	DP^{\#}v(t,x) \sim
	\sum_{r \in \mathbb{Z}_+, \, \ell \in -\mathbb{N}}
	\int e^{i(x-z) \cdot \xi +(t-t'')k-i\bar{\alpha} \cdot \xi}
	b_{r,\ell}(\bar{\alpha},t,x,t'',z,k,\xi)
	\Phi^{r,\ell}(k^{-\frac{1}{3}}\xi_1)v(t'',z)dt'' dz dk d\xi
\end{equation}
with 
\[
	b_{r,\ell}(\bar{\alpha},t,x,t'',z,k,\xi)
	= \int e^{i(y-z) \cdot (\eta-\xi)+i(t'-t'')(\tau-k)-i\bar{\alpha} \cdot (\eta-\xi)}a_{r,\ell}(x,t,\xi ,k)p^{\#}(\bar{\alpha},t',y)dt'dy d\tau d\eta.
\]

As for the composition of  the operator $DP^{\#}$ with the Fourier integral operator $J_1$ appearing in
\eqref{Q_N}, let us first note that
\begin{align}
	J_1 : \
	& \mathcal{D'}(\mathbb{R} \times \mathbb{R}^n) \to \mathcal{D}'(\mathbb{R}\times B)
	\nonumber
	\\
	& v(t',y) \mapsto J_1v(t,x)
	= \int e^{i\psi_1(x,\tau,\eta)+i\tau(t-t')-iy \cdot \eta} a_{J_1}(x,t',y)v(t',y) dt'dy d\tau d\eta
	\label{eq:J}
\end{align}
where $(\tau,\eta)$ are variables dual to $(t,x)$, $a_{J_1} \in S^0_{1,0}$, and the phase function $\psi_1$
is defined in a neighborhood of the base point as \cite[Equations 7.11 and 7.13]{MelroseTaylor85}
\begin{equation} \label{eq:psi1}
	\psi_1(x,\tau,\eta) =
	-\frac{|\eta'|^2}{2\tau}-\frac{|x'|^2\tau}{2} - 
	\left\{
		\begin{array}{cl}
			\frac{3}{2}
			(-\eta_1 \tau^{-\frac{1}{3}})^{\frac{3}{2}}
			\sign (\alpha \cdot \nu(x)),
			& \text{if } \alpha \cdot \nu(x) \ne 0,
			\\
			0,
			& \text{otherwise}.
		\end{array}
	\right.
\end{equation}
In \eqref{eq:psi1}, we have used the notation $x'=(x_2,...,x_n)$ for $x = (x_1,\ldots,x_n) \in B$
and similarly for $\eta \in \mathbb{R}^n$. Combining \eqref{eq:DPsharp} with \eqref{eq:J}, we obtain
\begin{align*}
	J_1DP^{\#}v(t,x)
	& \sim \sum_{r \in \mathbb{Z}_+, \, \ell \in -\mathbb{N}}
	\int e^{i\psi_1(x,\tau,\eta)+i\tau(t-t')-iy \cdot \eta} a_{J_1}(x,t',y) 
	\Big[ \int e^{i(y-z)\xi +i(t'-t'')k -i\bar{\alpha} \cdot \xi}
	\\
	& \hspace{3.495cm}
	b_{r,\ell}(\bar{\alpha},t',y,t'',z,k,\xi)
	\Phi^{r,\ell}(k^{-\frac{1}{3}}\xi_1) v(t'',z)
	dt''dz dk d\xi \Big]
	dt' dy d\tau d\eta,
\end{align*}
and we rewrite this as
\begin{align}
	J_1DP^{\#}v(t,x)
	& \sim \sum_{\ell \in -\mathbb{N}, \ r \in \mathbb{Z}_+}
	\int e^{i\psi_1(x,k,\xi)-i\bar{\alpha} \cdot \xi -itk -iz\xi -it'' k}
	\nonumber
	\\
	& \hspace{4.3cm}
	q_{N_{r,\ell}}(\bar{\alpha},x,t'',z,k,\xi)
	\Phi^{r,\ell}(k^{-\frac{1}{3}}\xi_1) v(t'',z)dt''dz dk d\xi
	\label{QDPA}
\end{align}
where
\begin{align*}
	q_{N_{r,\ell}}(\bar{\alpha},x,t'',z,k,\xi)
	= a_{J_1}\# b_{r,\ell}(\bar{\alpha},x,t'',z,k,\xi)
	& = \int e^{i\psi_1(x,\tau,\eta)-i\psi_1(x,k,\xi)+i(k-\tau)(t'-t) +i(\xi-\eta) \cdot y}
	\\
	& \hspace{1.2cm}
	a_{J_1}(x,t',y)
	b_{r,\ell}(\bar{\alpha},t',y,t'',z,k,\xi)
	dt'dy d\tau d\eta.
\end{align*}
In light of \eqref{eq:QNgood}, substituting \cite[p.295]{MelroseTaylor85}
\begin{equation}
	\mathcal{A}^{-1}\delta_{0}(t'',z)
	= \int e^{i\xi \cdot z+ik t''} \frac{1}{A_+(k^{-\frac{1}{3}}\xi_1)} dk d\xi
\end{equation}
for $v$ in \eqref{QDPA}, we get
\begin{align}
	Q_N\delta_{(0,\bar{\alpha})}(t,x)
	& \sim
	\sum_{r \in \mathbb{Z}_+, \, \ell \in -\mathbb{N}}
	\int e^{i\psi_1(x,k,\xi)-i\bar{\alpha} \cdot \xi -itk -iz \cdot \xi- it''k}
	\nonumber
	\\
	& \hspace{3cm}
	q_{N_{r,\ell}}(\bar{\alpha},x,t'',z,k,\xi)
	\Phi^{r,\ell}(k^{-\frac{1}{3}}\xi_1)
	\mathcal{A}^{-1}\delta_{0}(t'',z)dt''dz dk d\xi
	\nonumber
	\\
	& \sim
	\sum_{r \in \mathbb{Z}_+, \, \ell \in -\mathbb{N}}
	\int e^{i\psi_1(x,k,\xi)-i\bar{\alpha} \cdot \xi -itk}
	q_{N_{r,\ell}}(\bar{\alpha},x,k,\xi) \frac{\Phi^{r,\ell}(k^{-\frac{1}{3}}\xi_1)}{A_+(k^{-\frac{1}{3}}\xi_1)}dk d\xi
	\label{Q_N-composi}
\end{align}
where 
$q_{N_{r,\ell}}(\bar{\alpha},x,k,\xi)= a_J\# b_{r,\ell}(\bar{\alpha},x,k,\xi)$.

For the kernel $\kappa_{Q_N}$, we accordingly have 
\begin{equation*} \label{eq:QNasy}
	\kappa_{Q_N}(\alpha,t,x)
	\sim
	\sum_{r \in \mathbb{Z}_+, \, \ell \in -\mathbb{N}}
	\int e^{i\psi_1(x,k,\xi)-i\bar{\alpha} \cdot \xi -itk}
	q_{N_{r,\ell}}(\bar{\alpha},x,k,\xi)
	\frac{\Phi^{r,\ell}(k^{-\frac{1}{3}}\xi_1)}{A_+(k^{-\frac{1}{3}}\xi_1)}
	dk d\xi 
\end{equation*}
and, in virtue of Lemma~\ref{lemma:MT}, we deduce
(see \cite[p.298]{MelroseTaylor85})
\begin{equation} \label{eq:aqdouble}
	\eta(\alpha,x,k)
	\sim \sum_{r \in \mathbb{Z}_+, \, \ell \in - \mathbb{N}}
 	\int e^{i\psi_1(x,\xi,k)-i\bar{\alpha} \cdot \xi}
	q_{N_{r,\ell}}(\bar{\alpha},x,k,\xi)
	\frac{\Phi^{r,\ell}(k^{-\frac{1}{3}}\xi_1)}{A_+(k^{-\frac{1}{3}}\xi_1)}
	d\xi.
\end{equation}
Making the change of variable $\xi=k \zeta$ in \eqref{eq:aqdouble}, 
we therefore find \cite[p.298]{MelroseTaylor85}
\begin{equation} \label{eq:aqcv}
	\eta(\alpha,x,k)
	\sim \sum_{r \in \mathbb{Z}_+, \, \ell \in - \mathbb{N}}
 	\int e^{ik \psi_2(\alpha, x,\zeta)}
	q_{N_{r,\ell}}(\bar{\alpha},x,k,\zeta)
	\frac{\Phi^{r,\ell}(k^{\frac{2}{3}}\zeta_1)}{A_+(k^{\frac{2}{3}}\zeta_1)}
	d\zeta
\end{equation}
where $\psi_2(\alpha, x,\zeta) = \psi_1(x,1,\zeta)-\bar{\alpha} \cdot \zeta$, and with abuse of notation we have written
$q_{N_{r,\ell}}(\bar{\alpha},x,k,\zeta)$ for $k^n q_{N_{r,\ell}}(\bar{\alpha},x,k,\zeta)$.
In this case, $q_{N_{r,\ell}}$ is a symbol of order
$\frac{n}{2}+\frac{1}{6}-\frac{r}{3}-\frac{\ell}{3}+(\ell+1)_{-}$
so that 
\begin{equation} \label{eq:aqcvsym}
	q_{N_{r,\ell}}(\bar{\alpha},x,k,\zeta)
	\sim \sum_{q \in \mathbb{Z}_+}
	k^{\frac{n}{2}+\frac{1}{6}-q-\frac{r}{3}-\frac{\ell}{3}+(\ell+1)_{-}}
	q_{N_{q,r,\ell}}(\bar{\alpha},x,\zeta).
\end{equation}
Using \eqref{eq:aqcvsym} in \eqref{eq:aqcv}, we arrive at
\begin{equation} \label{eq:aqsimple}
	\eta(\alpha,x,k)
	\sim \sum_{q,r \in \mathbb{Z}_+ \atop \ell \in -\mathbb{N}}
	k^{\frac{n}{2}+\frac{1}{6}-q-\frac{r}{3}-\frac{\ell}{3}+(\ell+1)_{-}}
 	\int e^{ik\psi_2(\alpha, x,\zeta)}
	q_{N_{q,r,\ell}}(\bar{\alpha},x,\zeta)
	\frac{\Phi^{r,\ell}(k^{\frac{2}{3}}\zeta_1)}{A_+(k^{\frac{2}{3}}\zeta_1)}
	d\zeta. 
\end{equation}
In order to further simplify \eqref{eq:aqsimple}, we introduce the function
\begin{equation} \label{eq:Phiklasymp}
	\Psi^{r,\ell}(\tau)
	= e^{-\frac{i\tau^{3}}{3}} \int \frac{\Phi^{r,\ell}(s)}{A_+(s)}e^{-is\tau}ds.
\end{equation}
The symbolic behavior of $\Psi^{r,\ell}$ is as follows.

\begin{lemma} \label{eq:MT9.34} \emph{\cite[Lemma 9.34]{MelroseTaylor85}}
The function $\Psi^{r,\ell}$ defined in \eqref{eq:Phiklasymp} belongs to $S^{1+\ell-2r}(\mathbb{R})$, admits an asymptotic expansion
\begin{equation} \label{eq:Phikl}
	\Psi^{r,\ell}(\tau)
	\sim \sum_{j \in \mathbb{Z}_+} \alpha_{\ell,r,j} \tau^{1+\ell-2r-3j}
\end{equation}
as $\tau \to +\infty$, and is rapidly decreasing in the sense of Schwarz as $\tau \to -\infty$.
\end{lemma}

Rewriting \eqref{eq:Phiklasymp} as 
\[
	e^{\frac{i\tau^{3}}{3}}\Psi^{r,\ell}(\tau)
	=\int \frac{\Phi^{r,\ell}(s)}{A_+(s)}e^{-is\tau}ds
	=\widehat{\Big( \frac{\Phi^{r,\ell}}{A_+} \Big)}(\tau),
\]
and using $\mathcal{F}^{-1}$ to denote the inverse Fourier transform, we obtain
\begin{align}
	\frac{\Phi^{r,\ell}(k^{\frac{2}{3}}\zeta_1)}{A_+(k^{\frac{2}{3}}\zeta_1)}
	& = \mathcal{F}^{-1} \Big( \widehat{\Big( \frac{\Phi^{r,\ell}}{A_+} \Big)} (\tau) \Big) (k^{\frac{2}{3}}\zeta_1)
	= \mathcal{F}^{-1} \Big( e^{\frac{i\tau^{3}}{3}}\Psi^{r,\ell}(\tau) \Big) (k^{\frac{2}{3}}\zeta_1)
	\nonumber
	\\
	& = \int e^{ik^{\frac{2}{3}}\zeta_1 \tau} e^{\frac{i\tau^{3}}{3}} \Psi^{r,\ell}(\tau) d\tau 
	= k^{\frac{1}{3}}
	\int e^{ik \zeta_1 t + ik\frac{t^3}{3}}
	\Psi^{r,\ell}(k^{\frac{1}{3}}t)dt
	\label{function Psi}.
\end{align}
Using \eqref{function Psi} in \eqref{eq:aqsimple}, we finally conclude
\begin{equation}
	\eta(\alpha,x,k)
	\sim k^{\frac{1}{3}}
	\sum_{q,r \in \mathbb{Z}_+ \atop \ell \in -\mathbb{N}}
	k^{\frac{n}{2}+\frac{1}{6}-q-\frac{r}{3}-\frac{\ell}{3}+(\ell+1)_{-}} \,
	I_{q,r,\ell}(\alpha,x,k)\label{sum-aQN-2}
\end{equation}
where 
\begin{eqnarray}
	I_{q,r,\ell}(\alpha,x,k)
	= \int e^{ik \psi_3(\zeta ,t)}
	q_{N_{q,r,\ell}}(\bar{\alpha},x,\zeta)
	\Psi^{r,\ell}(k^{\frac{1}{3}}t)
	d\zeta dt
\end{eqnarray}
with the phase function given by
\begin{equation}
	\psi_3(\zeta ,t)
	= \psi_2(\alpha, x,\zeta) + t\zeta_1 + \frac{t^3}{3}.
\end{equation}

As shown in \cite{MelroseTaylor85}, the integrals $I_{q,r,\ell}(\alpha,k,x)$ can be treated using the stationary phase
method which results in
\begin{equation}\label{I-sum}
	I_{q,r,\ell}(\alpha,x,k)
	\sim k^{-\frac{1}{3}}
	\sum_{p \in \mathbb{Z}_+}
	k^{-\frac{n+1}{2}-\frac{2p}{3}}
	\, a_{p,q,r,\ell}(\alpha,x)
	\, (\Psi^{r,\ell})^{(p)}(k^{\frac{1}{3}} Z(\alpha,x))
	\, e^{ik \alpha \cdot x},
\end{equation}
and this leads into the following for the envelope $\eta^{\rm slow}(\alpha,k,x) = e^{-ik \alpha \cdot x} \eta(\alpha,k,x)$.

\begin{theorem} \label{thm:asymp-exp-local} \emph{\cite[Theorem 9.36]{MelroseTaylor85}}
In a vicinity of the shadow boundary $\{ x \in B : \alpha \cdot \nu(x) = 0 \}$, $\eta^{\rm slow}(\alpha,k,x)$ belongs to the
H\"{o}rmander class $S^{0}_{\frac{2}{3},\frac{1}{3}}$ and admits an asymptotic expansion
\begin{equation}
\label{eq:MT85Neumannlocal}
	\eta^{\rm slow}(\alpha,x,k)
	\sim
	\sum_{q,p,r \in \mathbb{Z}_+ \atop \ell \in -\mathbb{N}}
	a_{p,q,r,\ell} (\alpha,x,k)
\end{equation}
with
\[
	a_{p,q,r,\ell} (\alpha,x,k)
	= k^{-\frac{1}{3}-\frac{2p}{3}-q-\frac{r}{3}-\frac{\ell}{3}+(\ell+1)_-} \,
	b_{p,q,r,\ell} (\alpha,x) \,
	(\Psi^{r,\ell})^{(p)}(k^{\frac{1}{3}}Z(\alpha,x))
\]
where $b_{p,q,r,\ell}$ are complex-valued $C^{\infty}$ functions, $Z$ is a
real-valued $C^{\infty}$ function that is positive on the illuminated region $\{ x \in B : \alpha \cdot \nu(x) < 0 \}$,
negative on the shadow region $\{ x \in B : \alpha \cdot \nu(x) >0 \}$, and that vanishes precisely to first order at the shadow
boundary.
\end{theorem}

Under certain assumptions, Theorem~\ref{thm:asymp-exp-local} is in fact valid over the entire boundary $B$. This is given in the next theorem
where we use the notation $B^{\epsilon}_{\gtrless} = \{ x \in B: \alpha \cdot \nu(x) \gtrless \epsilon\}$.

\begin{theorem} \label{thm:asymp-exp-global}
Assume there exists $\epsilon \in (0,1)$ such that on $B^{\epsilon}_{<}$ the envelope $\eta^{\rm slow}$ belongs to
$S^0_{1,0}(B^{\epsilon}_{<} \times (0,\infty))$
and admits an asymptotic expansion
\begin{equation} \label{eq:ilasymp}
	\eta^{\rm slow}(\alpha,x,k)
	\sim \sum_{j \in \mathbb{Z}_+} k^{-j} \, a_j(\alpha,x), 
	\quad
	\text{as } k \to \infty,
\end{equation}
and it is rapidly decreasing in the sense of Schwarz on $B^{\epsilon}_{>}$ as $k \to -\infty$. Then 
$\eta^{\rm slow} \in S^{0}_{\frac{2}{3},\frac{1}{3}}(B \times (0,\infty))$ and the asymptotic expansion
\eqref{eq:MT85Neumannlocal} is valid over the entire boundary $B$.
\end{theorem}

The proof of Theorem~\ref{thm:asymp-exp-global} follows the same lines as in the proof of \cite[Corollary 5.3]{DominguezEtAl07}
(see also \cite[Theorem 3.1 and Corollary 2.1]{EcevitReitich09}) and is based on the standard \emph{matching of asymptotic expansions
technique} (see e.g. \cite{DominguezEtAl07} and the references therein). The expansion \eqref{eq:ilasymp} related to Neumann problem is
similar to the one given in \cite[Equation 1.15]{MelroseTaylor85} for the Dirichlet case. Furthermore, using the references provided in
\cite{DominguezEtAl07} (see the proof of Corollary 5.3) we can deduce that, for the two-dimensional Neumann boundary value
problem, $\eta^{\rm slow}$ decays exponentially in $B^{\epsilon}_{>}$ as $k \to -\infty$ which implies the assumption
of its rapid decay in the sense of Schwarz in Theorem~\ref{thm:asymp-exp-global}.


\section{Auxiliary results} \label{sec:auxiliary}

Here we provide auxiliary results used in the proofs.

\begin{lemma} \emph{\cite[Lemma 14]{Ecevit18}}
\label{lemma:generalderivatives}
Let $a(s,k) = k^{\theta} \, b(s) \, \varphi(k^{\omega}Z(s))$ where $b,\varphi$ and $Z$ are smooth functions, $b$ and $Z$ are periodic,
and $\theta \in \mathbb{R} \backslash \mathbb{N}$ and $\omega \in \mathbb{R} \backslash \Zplus$. Then
\begin{align*}
	|D_s^{n}D_{k}^{m} \, a(s,k)|
	& \lesssim
	k^{\theta-m}
	\sum_{j=0}^{n+m}
	k^{j \omega} 
	|\varphi^{(j)} (k^{\omega}Z(s))|
\end{align*}
for all $n,m \in \Zplus$ and all $k >0$.
\end{lemma}

\begin{theorem} \label{thm:pae} \emph{\cite[Corollary 3.12]{Schwab98}}
Given a function $f \in C^{\infty}([a,b])$ and $n \in \mathbb{Z}_+$, there exists a constant
$C_n > 0$ such that
\begin{equation*}
	\inf_{p \in \mathbb{P}_d}
	\Vert f-p \Vert_{L^2[(a,b )]}
	\le C_n \left[
		\int_{a}^{b}
		\left| D^{n} f(s) \right|^{2}
		\left( s-a \right)^{n}
		\left( b-s \right)^{n}
		ds
	\right]^{\frac{1}{2}} \, d^{-n}
\end{equation*}
for all $d \in \mathbb{N}$ with $d+1 \ge n$.
\end{theorem}

\begin{lemma} \label{lemma:exactintegral} \emph{\cite[Lemma 14]{Ecevit18}}
Suppose that either $[\alpha,\beta] \subseteq [t_{1},t_{2}] \subseteq (c,d)$ or
$[\alpha,\beta] \cap (t_{1},t_{2}) = \varnothing$ and $[c,d] \subseteq (t_1,t_2)$.
Then, for any $a,b \in \mathbb{R}$, $n \in \mathbb{N} \cup \{ 0 \}$, $m \in \mathbb{N}$,
there holds
\begin{align*}
	\int_{\alpha}^{\beta} \dfrac{(s-a)^{n} \, (b-s)^{n}}{(s-c)^{m} \, (d-s)^{m}} \, ds
	= \sum_{0 \le p,q \le n \atop 1 \le j \le m}
	\binom{2m-j-1}{m-j} \binom{n}{p} \binom{n}{q}
	\dfrac{(-1)^{n} \, \mathcal{F}(\alpha,\beta;a,b;c,d;n,p,q;j)}{(d-c)^{2m-j}}.
\end{align*}
Here we have
\[
	\mathcal{F}(\alpha,\beta;a,b;c,d;n,p,q;j)
	= \left( c-a \right)^{p} \left( c-b \right)^{q} \log \left( \dfrac{\beta-c}{\alpha-c} \right)
	+ \left( a-d \right)^{p} \left( b-d \right)^{q} \log \left( \dfrac{d-\alpha}{d-\beta} \right)
\]
when $2n-(p+q+j) = -1$, and
\begin{align*}
	& \mathcal{F}(\alpha,\beta;a,b;c,d;n,p,q;j)
	\\
	& \hspace{2cm}
	= \dfrac{\left( c-a \right)^{p} \left( c-b \right)^{q}}{2n-(p+q+j)+1}
	\left[ \left( \beta-c \right)^{2n-(p+q+j)+1} - \left( \alpha-c \right)^{2n-(p+q+j)+1} \right]
	\\
	& \hspace{2cm}
	+ \dfrac{\left( a-d \right)^{p} \left( b-d \right)^{q}}{2n-(p+q+j)+1}
	\left[ \left( d-\alpha \right)^{2n-(p+q+j)+1} - \left( d-\beta \right)^{2n-(p+q+j)+1} \right]
\end{align*}
when $2n-(p+q+j) \ne -1$.
\end{lemma}

\section*{Acknowledgements}

A. Anand gratefully acknowledges the support by Science \& Engineering Research Board through File No MTR/2017/000643.
Y. Boubendir's work was supported by the NSF through Grants DMS-1720014 and DMS-2011843.
F. Ecevit is supported by the Scientific and Technological Research Council of Turkey through grant T\"{U}B\.{I}TAK-1001-117F056.

\bibliographystyle{abbrv}
\bibliography{AnandBoubendirEvevitLazergui2020.bib}

\end{document}